\documentclass[a4paper,10pt]{article}
\usepackage{makeidx}
\usepackage[mathscr]{eucal}
\usepackage{amssymb, amsfonts, amsmath, amsthm, graphicx, cases, color, tikz, url} 
\usepackage{here}
\usepackage{tasks}
\usepackage{tikz-cd}
\usepackage{hyperref}

\newtheorem{definition}{Definition}[section]
\newtheorem{proposition}{Proposition}[section]
\newtheorem{theorem}{Theorem}[section]

\newtheorem{lemma}[proposition]{Lemma}
\newtheorem{remark}{Remark}[section]

\newcommand{\R}{\mathbb{R}}

\renewcommand{\S}{\mathcal{S}}

\newcommand{\dif}{\mathrm{d}}

\numberwithin{equation}{section}

\numberwithin{equation}{section}

\title{\textbf{Solutions to Sobolev Supercritical Nonlinear Schrödinger Equations on an Annulus via a Hopf Reduction Method}}
\author{Jian Liang, Hua-Yang Wang\footnote{Corresponding author.}} 
\date{}
\begin{document}

\bibliographystyle{abbrv}

\maketitle

\begin{abstract}
   This paper investigates the existence of positive solutions with a prescribed mass for nonlinear Schrödinger equations on an annulus, possibly in the Sobolev supercritical regime. 
   A reduction method based on the Hopf fibration is used to transform the problem into a lower-dimensional one. 
   We obtain a new mass critical threshold and we show that in the new mass subcritical or critical regimes there exists a positive solution which corresponds to a global minimizer, while in the mass supercritical regime, there exists two positive solutions which correspond to a local minimizer and a mountain pass solution respectively. 
   Some other problems are also discussed in this paper.
\end{abstract}



\section{Introduction}

\subsection{Motivations and backgrounds}
Let $\Omega \subset \mathbb{R}^N$ be a bounded domain with $N \geq 3$. 
The time-dependent nonlinear Schrödinger equation
\begin{equation}\label{eq:1.1}
    \begin{cases}
        i \frac{\partial \psi}{\partial t}+\Delta \psi+ |\psi|^{p-2} \psi =0, \quad & (t, x) \in \mathbb{R} \times \Omega, \\
        \psi(t, x)=0, \quad & (t, x) \in \mathbb{R} \times \partial \Omega, \\
        \psi(0, x)=\psi_0(x), \quad & x \in \Omega,
    \end{cases}
\end{equation}
arises naturally in various physical phenomena, such as nonlinear optics \cite{Agrawal2000}, plasma physics \cite{Burke1994} and Bose-Einstein condensates \cite{Gadi2001}. 
If we look for standing waves of the form $\psi(t, x)=e^{-i \lambda t} u(x)$, then $u$ solves the following elliptic problem
\begin{equation}\label{eq:DP}
    \begin{cases}
        -\Delta u + \lambda u= |u|^{p-2} u & \text { in } \Omega, \\ 
        u=0 & \text { on } \partial \Omega.
    \end{cases}
\end{equation}

A typical way to investigate the elliptic problem \eqref{eq:DP} is the variational method.
Given $\lambda \in \mathbb{R}$, solutions to \eqref{eq:DP} can be viewed as critical points of the action functional
\begin{equation*}
    I(u)=\frac{1}{2} \int_\Omega |\nabla u|^2 \, \dif x + \frac{\lambda}{2} \int_\Omega |u|^2 \, \dif x - \frac{1}{p} \int_\Omega |u|^p \, \dif x.
\end{equation*}
It is known that for $2 < p < 2^* := \frac{2N}{N-2}$ (the Sobolev critical exponent), \eqref{eq:DP} has a positive solution for every $\lambda > -\lambda_1(\Omega)$, where $\lambda_1(\Omega)$ is the first eigenvalue of $-\Delta$ in $H^1_0(\Omega)$ \cite{Willem1996}. 
However, if $p \geq 2^*$ and $\Omega$ is star-shaped, then the Pohozaev identity \cite{Pohozaev1965}
\begin{equation*}
    \left(\frac{1}{2}-\frac{1}{p}\right) \int_\Omega |\nabla u|^2 \, \dif x + \frac{\lambda}{2} \int_\Omega |u|^2 \, \dif x = \frac{1}{2N} \int_{\partial \Omega} (x \cdot \nu) |\nabla u|^2 \, \dif S
\end{equation*}
implies that \eqref{eq:DP} has no nontrivial solutions for any $\lambda \geq 0$. 
Therefore, for any star-shaped domain, we cannot hope to seek a solution to \eqref{eq:DP} when $\lambda \geq 0$ with $p$ being Sobolev critical or supercritical. 

Nevertheless, if $\Omega$ possesses a nontrivial topology, such as containing `a hole', the situation becomes substantially different.
The case $p = 2^*$ was studied by the seminal work of Bahri and Coron \cite{Bahri1988}. 
They proved that if $\Omega$ is a bounded domain with nontrivial $\mathbb{Z}_2$ reduced homology group, then \eqref{eq:DP} has a positive solution for every $\lambda > 0$.
The problem is much more complicated if $p > 2^*$. 
As the Sobolev embedding $H^1_0(\Omega) \hookrightarrow L^p(\Omega)$ does not hold for $p > 2^*$, the functional $I$ is not well-defined in $H^1_0(\Omega)$, and hence the variational methods cannot be used directly. 
However, in specific domains, we can still establish the existence of positive solutions to \eqref{eq:DP} for certain Sobolev supercritical ranges.
When $\Omega = \{x \in \mathbb{R}^N: a < |x| < b, 0 < a < b < +\infty\}$ is an annulus, due to the rich symmetry structure, a special coordinate based on the Hopf fibration was introduced by Ruf and Srikanth in \cite{Ruf2010} for $N = 4$ to deal with the following singularly perturbed problem
\begin{equation}\label{eq:SPP}
    \begin{cases}
        -\varepsilon^2 \Delta u + u= u^{p-1} & \text { in } \Omega, \\
        u =0, & \text { on } \partial \Omega,
    \end{cases}
\end{equation} 
and the problem \eqref{eq:SPP} can be reduced to the following problem in $\R^3$ with a weight
\begin{equation}\label{eq:1.3}
    \begin{cases}
        -\varepsilon^2 \Delta v + \frac{1}{2|x|} v= \frac{1}{2|x|} v^{p-1} & \text { in } \widetilde{\Omega}, \\
        v =0, & \text { on } \partial \widetilde{\Omega},
    \end{cases}
\end{equation}       
where $\widetilde{\Omega}=\left\{x \in \mathbb{R}^3: a^2 / 2< |x| <b^2 / 2\right\}$ is also an annulus (the details of the reduction will be discussed in Section 2).
Applying methods and results in \cite{Ni1995} to \eqref{eq:1.3}, they were able to show that there exist a sequence of solutions to \eqref{eq:SPP} concentrating in an $S^1$ orbit as $\varepsilon \rightarrow 0^+$ for every $p \in (2,6)$.
Notice that $6$ is larger than the Sobolev critical exponent $4$ in the original problem. 
In summary, the reduction method provides a way to deal with some Sobolev supercritical problems. 
Such a method has been generalized in \cite{Ruf2014,Manna2014,Santra2016,Manna2022,Liu2022} to investigate solutions to \eqref{eq:DP} in certain Sobolev supercritical ranges.

There is another perspective to study the problem \eqref{eq:DP}. 
Two conserved quantities of \eqref{eq:1.1} are important, that is, the energy
\begin{equation*}
    E(\psi)=\frac{1}{2} \int_\Omega |\nabla \psi|^2 \, \dif x - \frac{1}{p} \int_\Omega |\psi|^p \, \dif x
\end{equation*}
and the mass
\begin{equation*}
    M(\psi)=\int_\Omega |\psi|^2 \, \dif x.
\end{equation*}
Therefore, investigating solutions with a prescribed mass is a natural endeavor; in other words, we aim to find solutions to \eqref{eq:DP} satisfying the constraint
\begin{equation*}
    \int_\Omega |u|^2 \, \dif x = c^2
\end{equation*}
for some fixed $c > 0$.
We shall call such problems the \textbf{prescribed mass problem} and corresponding solutions \textbf{normalized solutions}.
The stability of solutions to \eqref{eq:1.1} is closely related to the existence of normalized solutions to \eqref{eq:DP} \cite{Cazenave1982}. 
The prescribed mass problem can be restated as
\begin{equation}\label{eq:OP}
    \begin{cases}
        -\Delta u + \lambda u= |u|^{p-2} u & \text { in } \Omega, \\ 
        u=0 & \text { on } \partial \Omega, \\
        \int_\Omega |u|^2 \, \dif x = c^2.
    \end{cases}
\end{equation}
Solutions to \eqref{eq:OP} can be viewed as critical points of the energy functional $E: H^1_0(\Omega) \rightarrow \mathbb{R}$ restricted to the manifold
\begin{equation*}
    S_c = \{u \in H^1_0(\Omega): \int_\Omega |u|^2 \, \dif x = c^2\},
\end{equation*}
where $E$ is defined by
\begin{equation*}
    E(u)=\frac{1}{2} \int_\Omega |\nabla u|^2 \, \dif x - \frac{1}{p} \int_\Omega |u|^p \, \dif x
\end{equation*}
and $\lambda$ appears here as an unknown Lagrange multiplier. 
According to the different variational structures of $E$ on $S_c$, the mass critical exponent
\begin{equation*}
    2^{\#}_N := 2 + \frac{4}{N}
\end{equation*}
plays an important role in the prescribed mass problem. 
Indeed, by the Gagliardo-Nirenberg inequality (see \eqref{eq2.5}), the functional $E$ is bounded from below on $S_c$ if $2 < p < 2^{\#}_N$, while $E$ is unbounded from below on $S_c$ if $ 2^{\#}_N < p < 2^*$. If $p = 2^{\#}_N$, the lower bounds of $E$ on $S_c$ depends on the value of $c$.

For $2 < p < 2^{\#}_N$, it is not hard to prove the existence of a positive solution to \eqref{eq:OP} by a minimizing progress since the Sobolev embedding $H^1_0(\Omega) \hookrightarrow L^p(\Omega)$ is compact (for example, see \cite{Pierotti2017}).  
For $2^{\#}_N < p < 2^*$, Noris, Tavares and Verzini \cite{Noris2014} first studied \eqref{eq:OP} when $\Omega$ is a unit ball. 
They proved the existence of at least two positive solutions for $c > 0$ small enough. 
Pierotti and Verzini \cite{Pierotti2017} investigated \eqref{eq:OP} if $\Omega$ is a general bounded domain via a different method. 
They showed that for $c > 0$ small enough, there exist at least two positive solutions, one corresponds to a local minimizer of $E$ on $S_c$ and the other one is of mountain pass type \footnote{The proof about the existence of the mountain pass type solution in \cite{Pierotti2017} contained a gap, which was filled in \cite{Pierotti2025}.}.
The result was then generalized by Pierotti, Verzini and Yu \cite{Pierotti2025} to deal with the Sobolev critical case, where an energy estimate is performed to overcome the lack of compactness. 
According to the above discussion, it is natural to consider the following open problem:

\textbf{If $\Omega$ is a bounded domain with a nontrivial topology, does \eqref{eq:OP} have a positive solution for $p > 2^*$?} 

We remark that a related open problem was proposed by Rabinowitz for the Dirichlet problem \eqref{eq:DP}, which was stated by Brezis in \cite{Brezis1986}.
The above problem can be seen as a prescribed mass version of Rabinowitz's open problem. 
We shall give an affirmative answer to it when $\Omega \subset \R^N$ is an annulus with $N = 4,8,16$.
Indeed, we have studied a more general problem for dimensions $N \geq 3$ in the intermediate steps and we will address why it cannot be applied to other dimensions in the final section.

\subsection{Main results}

Inspired from the Hopf fibration in \cite{Ruf2010, Ruf2014}, the problem \eqref{eq:OP} with $N = 4,8,16$ can be reduced to the following problem in a lower dimension case
\begin{equation}\label{eq:RPP}
    \begin{cases}
        -\Delta v + \frac{\lambda}{2|x|} v= \frac{1}{2|x|} v^{p-1} & \text { in } D, \\
        v(x) = 0, & \text { on } \partial D, \\
        \int_D \frac{1}{2|x|} v^2 \, \dif x = \hat{c}^2,
    \end{cases}
\end{equation}
where $D \subset \mathbb{R}^{M}$ is also an annulus, $M = 3,5,9$ and $\hat{c}$ is another constant determined by $c$ (see details in Section 2). 
Therefore, to obtain the existence results for \eqref{eq:OP}, it suffices to study the existence of positive solutions to \eqref{eq:RPP}. 
More generally, we will study the existence and stability of positive solutions for a more general class of equations
\begin{equation}\label{eq:RGP}
    \begin{cases}
        -\Delta u + \lambda |x|^\mu u= |x|^\mu f(u) & \text { in } D, \\ 
        u(x)=0, & \text { on } \partial D, \\
        \int_D |x|^\mu u^2 \, \dif x = c^2,
    \end{cases}
\end{equation}
where $D=\{x \in \mathbb{R}^{M}: a<|x|<b, 0<a<b<+\infty\}$ is an annulus, $M\geq 3$, $c > 0$ is a prescribed constant, $\mu \in \mathbb{R}$. Solutions of \eqref{eq:RGP} can be viewed as critical points of the energy functional
\begin{equation*}
    E(u)=\frac{1}{2} \int_D |\nabla u|^2 \, \dif x - \int_D |x|^\mu F(u) \, \dif x
\end{equation*}
on the manifold
\begin{equation*}
    {\S}_c = \{u \in H^1_0(D): \int_D |x|^\mu u^2 \, \dif x = c^2\},
\end{equation*}
where $F(u) = \int_0^u f(s) \, \dif s$.
That is, every $u \in {\S}_c$ satisfying $E'|_{{\S}_c}(u) = 0$ is a solution to \eqref{eq:RGP} for some $\lambda \in \mathbb{R}$.
It is worth noting that when we set $\mu = -1$ and $f(u) = |u|^{p-2} u$,  \eqref{eq:RGP} goes back to \eqref{eq:RPP} in $M = 3,5,9$ up to a constant.

We will assume the following conditions on the nonlinearity $f$:

($\mathbf{A_1}$) $f \in C^1([0,+\infty)), f(s)>0$ for $s>0$, and $\lim\limits_{s \rightarrow 0^+} \frac{f(s)}{s} = 0$.

($\mathbf{A_2}$) There exists some $2 < q < \frac{2M}{M-2}$ and $a_0 > 0$ such that
\begin{equation*}
\lim _{s \rightarrow+\infty} \frac{f(s)}{s^{q-1}} = a_0.
\end{equation*}

($\mathbf{A_3}$) There exists a constant $a_1 > a_0 (q-1)$ and some $R >0$ such that
\begin{equation*}
    a_1 |s|^{q-2} \leq f'(s), \quad \text{ for } s \geq R,
\end{equation*}
Without loss of generality, we will also assume that $f$ is extended to the whole real line by setting $f(s) = 0$ for $s < 0$.

Under the above assumptions, the functional $E$ is well-defined and of class $C^2$ on $\S_c$. 

\begin{remark}
    Notice that there are many nonlinearities satisfying the above conditions. 
    For example:
    \begin{enumerate}
        \item[$(i)$] The pure power nonlinearity $f(u) = C u^{p-1}$, where $C > 0$ and $2 < p < \frac{2M}{M-2}$;
        \item[$(ii)$] The mixed power nonlinearity $f(u) = C_1 u^{p_1-1} + C_2 u^{p_2-1}$, where $C_1, C_2 >0$ and $2 < p_1,p_2 < \frac{2M}{M-2}$.
    \end{enumerate}
\end{remark}

The main result of this paper is the following theorem:
\begin{theorem}\label{thm1.2}
    Let $M \geq 3$.
    Assume that $f$ satisfies the assumptions ($\mathbf{A_1}$) and ($\mathbf{A_2}$). 
    Denote the mass critical exponent in dimension $M$ by $2^{\#}_M = 2 + \frac{4}{M}$, then the following conclusions hold:
    \begin{enumerate}
        \item[$(i)$] If $2 < q < 2^{\#}_M$, then for every $c > 0$, \eqref{eq:RGP} has a positive solution, which corresponds to a global minimizer of $E$ on ${\S}_c$;
        \item[$(ii)$] If $q = 2^{\#}_M$, then there exists some $c_1 > 0$ such that for every $0 < c < c_1$, \eqref{eq:RGP} has a positive solution, which corresponds to a global minimizer of $E$ on ${\S}_c$;
        \item[$(iii)$] If $2^{\#}_M < q < \frac{2M}{M-2}$, then there exists some $c_2>0$ such that if $0< c < c_2$, then \eqref{eq:RGP} has a local minimizer of $E$ on ${\S}_c$. 
        If in addition $f$ satisfies the assumption ($\mathbf{A_3}$), then there exists a second positive solution, which is of mountain pass type for $E$ on ${\S}_c$.
    \end{enumerate}
\end{theorem}

In particular, if we consider the case $M = 3,5,9$, by reversing the reduction process, we can obtain the existence of positive solutions for the supercritical problem \eqref{eq:OP} in dimension $4, 8$ and $16$.

\begin{theorem}\label{cor1.3}
    Assume $\Omega \subset \mathbb{R}^{N}$ is an annulus and $N = 4,8,16$. Denote the reduced Sobolev critical exponent by
    \begin{equation*}
        \tilde{2}^* =
        \begin{cases}
            6 & \text{ if } N = 4, \\
            \frac{10}{3} & \text{ if } N = 8, \\
            \frac{18}{7} & \text{ if } N = 16,
        \end{cases}
    \end{equation*}
    and the reduced mass critical exponent by
    \begin{equation*}
        \tilde{2}^\# =
        \begin{cases}
            \frac{10}{3} & \text{ if } N = 4, \\
            \frac{14}{5} & \text{ if } N = 8, \\
            \frac{22}{9} & \text{ if } N = 16.
        \end{cases}
    \end{equation*}
    Then the following conclusions hold:
    \begin{enumerate}
        \item[$(i)$] If $2 < p < \tilde{2}^\#$, then for every $c > 0$, \eqref{eq:OP} has a positive solution, which corresponds to a global minimizer of $E$ on ${\S}_c$;
        \item[$(ii)$] If $p = \tilde{2}^\#$, then there exists some $c_3 > 0$ small such that for every $0 < c < c_3$, \eqref{eq:OP} has a positive solution, which corresponds to a global minimizer of $E$ on ${\S}_c$;
        \item[$(iii)$] If $\tilde{2}^\# < p < \tilde{2}^*$, then there exists some $c_4>0$ such that if $0< c < c_4$, then \eqref{eq:OP} has two positive solutions, one corresponds to a local minimizer of $E$ on ${\S}_c$ and the other corresponds to a mountain pass level for $E$ on ${\S}_c$. 
    \end{enumerate}
\end{theorem}

\begin{remark}
    In other dimensions $N \neq 4,8,16$, the reduction method based on the Hopf fibration is still valid. 
    However, the reduced problem will be defined on a warped product manifold rather than an annulus in the Euclidean space. 
    This will bring some additional difficulties in studying the existence of solutions to the reduced problem. 
    We will discuss this issue in the final section.
\end{remark}

\begin{remark}
    Notice that the original Sobolev critical exponent $2^*$ in these cases are given by
    \begin{equation*}
        2^* = 
        \begin{cases}
            4 & \text{ if } N = 4, \\
            6 & \text{ if } N = 8, \\
            \frac{10}{3} & \text{ if } N = 16.
        \end{cases}
    \end{equation*}
    So, the reduced Sobolev critical exponent $\tilde{2}^*$ is larger than the original Sobolev critical exponent $2^*$. 
    Thus, the results in Theorem \ref{cor1.3} are indeed new results for some Sobolev supercritical nonlinearities.
\end{remark}

\subsection{Plan of the paper}

In Section 2, we will first introduce the reduction method based on Hopf fibration in detail. From the reduction method, we can give a full description of the reduced problem \eqref{eq:RPP}. 
Then we will introduce some useful results, such as the Gagliardo-Nirenberg inequality and some properties of the limiting equation. 

In Section 3 the case $2 < q \leq 2^{\#}_M$ will be discussed. We will use the Gagliardo-Nirenberg inequality to show that the energy functional $E$ is bounded from below on ${\S}_c$ and coercive, so the existence of a global minimizer can be obtained by the Ekeland variational principle and a compactness lemma.

In Section 4 the case $2^{\#}_M < q < \frac{2M}{M-2}$ will be discussed. We will use the Gagliardo-Nirenberg inequality again to show that $E$ has a local minimizer on ${\S}_c$ for $c > 0$ small. With the local minimizer, we can construct a mountain pass geometry for $E$ on ${\S}_c$. To find a mountain pass type solution, we will use the monotonicity trick and a blow-up analysis to show the boundedness of a special Palais-Smale sequence. Therefore, a mountain pass type solution will be obtained by a compactness lemma.

Finally, in Section 5 we give some comments and discuss some open problems based on this work, which will be discussed in the future. 

An appendix is also given to show the full details of the Hopf reduction process.

In this paper, we will use the following notations:
\begin{itemize}
    \item $C, C_i$ denote positive constants that may vary from line to line;
    \item $B_R(x_0)$ denotes the open ball in $\mathbb{R}^M$ centered at $x_0$ with radius $R > 0$;
    \item $H^1_0(D)$ denotes the usual Sobolev space with the norm $\|u\|_{H^1_0} = (\int_D |\nabla u|^2 \, \dif x)^{1/2}$, the dual space of $H^1_0(D)$ is denoted by $H^{-1}(D)$;
    \item $L^p(D)$ denotes the usual Lebesgue space with the norm $\|u\|_{L^p} = (\int_D |u|^p \, \dif x)^{1/p}$ for $1 \leq p < +\infty$;
    \item $\sim$ means that there exist some positive constants such that the two sides are bounded by each other up to these constants, i.e., $A \sim B$ means that there exist $C_1, C_2 > 0$ such that $C_1 A \leq B \leq C_2 A$;
    \item $o_n(1)$ denotes a quantity depending on $n$ such that $\lim\limits_{n \rightarrow +\infty} o_n(1) = 0$.
\end{itemize}
\section{Preliminaries}

\subsection{The Hopf reduction method}

We consider the following problem
\begin{equation}\label{eq2.1}
    \begin{cases}
        -\Delta u+ \lambda |x|^\alpha u=|x|^\alpha f(u) & \text { in } A, \\ 
        u=0 & \text { on } \partial A, \\
        \int_A |x|^\alpha u^2 \, \dif x = c^2,
    \end{cases}
\end{equation}
where $c > 0$, $\alpha \in \mathbb{R}$, and $A=\left\{x \in \mathbb{R}^{N}: a<|x|<b, 0<a<b<+ \infty\right\}$ is an annulus in $\mathbb{R}^{N}$ with $N = 4, 8, 16$. 
We employ a transformation based on the Hopf fibration, a method originally introduced by Ruf and Srikanth in \cite{Ruf2010} for $N = 4$ and later generalized in \cite{Ruf2014} for $N = 8, 16$.

We regard the annulus as a product space $(a, b) \times S^{N-1}$. The Hopf fibration is illustrated by the following diagrams:
\begin{equation*}
    \begin{tikzcd}[column sep=large, row sep=large]
    S^{1} \arrow[r, "g"] & S^{3} \arrow[d, "\pi"] \\
    & \mathbb{CP}^{\,1}
    \end{tikzcd}
    \qquad
    \begin{tikzcd}[column sep=large, row sep=large]
    S^{3} \arrow[r, "g"] & S^{7} \arrow[d, "\pi"] \\
    & \mathbb{HP}^{\,1}
    \end{tikzcd}
    \qquad
    \begin{tikzcd}[column sep=large, row sep=large]
    S^{7} \arrow[r, "g"] & S^{15} \arrow[d, "\pi"] \\
    & \mathbb{OP}^{\,1}
    \end{tikzcd}
\end{equation*}
where $\mathbb{CP}^{\,1} \cong S^3 / S^1$, $\mathbb{HP}^{\,1} \cong S^7 / S^3$, and $\mathbb{OP}^{\,1} \cong S^{15} / S^7$ denote the complex, quaternionic, and octonionic projective lines, respectively. Here, $g$ represents the group action defined as follows:
\begin{itemize}
    \item for $N = 4$, the action of $S^1$ on $S^3$ is defined by the multiplication of unit complex numbers;
    \item for $N = 8$, the action of $S^3$ on $S^7$ is defined by the multiplication of unit quaternions;
    \item for $N = 16$, the action of $S^7$ on $S^{15}$ is defined by the left-multiplication of unit octonions.
\end{itemize}
The map $\pi$ is the corresponding quotient map (refer to \cite{Petersen2016} for further details). 

Let $(r, \theta)$ denote the spherical coordinates in $\mathbb{R}^{N}$, where $r = |x|$ and $\theta \in S^{N-1}$. 
These group actions are free, meaning $g \theta \neq \theta$ for any $\theta \in S^{N-1}$ and any $g \neq e$ in the respective group. 
We focus on functions in $H_0^1(A)$ that are invariant under these group actions, defined as
\begin{equation*}
    H_{0,G}^1(A) = \{u \in H_0^1(A): u(r, g \theta) = u(r, \theta),~\forall g \in G\},
\end{equation*}
where $G$ is $S^1, S^3$, or $S^7$ for $N = 4, 8, 16$ respectively. 
Each $u \in H_{0,G}^1(A)$ can be identified with a function defined in $H^1_0((a, b) \times \mathbb{FP}^{\,1})$, where $\mathbb{FP}^{\,1}$ represents $\mathbb{CP}^{\,1}, \mathbb{HP}^{\,1}$, or $\mathbb{OP}^{\,1}$. 
Specifically, there exists a function $v$ defined on $(a, b) \times \mathbb{FP}^{\,1}$ such that
\begin{equation*}
    u(r, \theta) = v(r, \pi(\theta)).
\end{equation*}
This identification allows us to reduce problem \eqref{eq2.1} to a lower-dimensional problem.
Setting $M = 3, 5, 9$ for $N = 4, 8, 16$ respectively, the reduced problem is defined on an annulus $D = \{x \in \mathbb{R}^{M}: a^2 / 2 < |x| < b^2 / 2\}$ in $\mathbb{R}^{M}$
\begin{equation}\label{eq:2.2}
    \begin{cases}
        -\Delta v + \lambda (2 |x|)^{\frac{\alpha}{2}-1} v= (2 |x|)^{\frac{\alpha}{2}-1} f(v) & \text { in } D, \\
        v=0 & \text { on } \partial D, \\
        \int_D (2 |x|)^{\frac{\alpha}{2}-1} v^2 \, \dif x = \hat{c}^2,
    \end{cases}
\end{equation}
where $\hat{c} = c \omega_{N-M}^{-\frac{1}{2}}$. Here $\omega_{N-M}$ denotes the surface area of the unit sphere $S^{N-M}$ in $\mathbb{R}^{N-M}$. The details of this reduction process are provided in the Appendix. Note that this problem can be rewritten in the form of \eqref{eq:RGP} with $\mu = \frac{\alpha}{2} - 1$.

\begin{remark}
    For the case $N = 4$, the reduction process can be performed explicitly.
    Consider the following coordinate system for the annulus $A=\{x \in \mathbb{R}^4: a<|x|<b\}$
    \begin{equation*}
        \begin{array}{ll}
            x_1=r \sin \theta_1 \sin \theta_3, \\
            x_2=r \cos \theta_1 \sin \theta_3, \\
            x_3=r \sin \theta_2 \cos \theta_3, \\
            x_4=r \cos \theta_2 \cos \theta_3, 
        \end{array}
    \end{equation*}
    where $a<r<b$, $\theta_1, \theta_2 \in [0, 2\pi)$ represent the polar angles in the $x_1 x_2$-plane and the $x_3 x_4$-plane respectively, and $\theta_3 \in [0, \pi/2)$ represents the angle between the planes $x_1 x_2$ and $x_3 x_4$. If we consider functions in $H_0^1(A)$ invariant under the group action
    \begin{equation*}
        T_\tau\left(r, \theta_1, \theta_2, \theta_3\right)=\left(r, \theta_1+\tau, \theta_2+\tau, \theta_3\right), \quad \tau \in[0,2 \pi),
    \end{equation*}
    then problem \eqref{eq2.1} reduces to the following problem in $D \subset \mathbb{R}^{3}$:
    \begin{equation*}
        \begin{cases}
            -\Delta v+ \lambda |x|^{\frac{\alpha}{2}-1} v=|x|^{\frac{\alpha}{2}-1} f(v) & \text { in } D, \\ 
            v=0 & \text { on } \partial D, \\
            \int_D |x|^{\frac{\alpha}{2}-1} v^2 \, \dif x = \hat{c}^2,
        \end{cases}
    \end{equation*}
    where $D=\left\{x \in \mathbb{R}^3: a^2 / 2<|x|<b^2 / 2\right\}$ and $\hat{c} = (\frac{c^2}{2 \pi} \cdot 2^{-\frac{\alpha}{2}})^{\frac{1}{2}} > 0$. This is a special case of \eqref{eq:2.2} with $M = 3$.
\end{remark}

\subsection{Gagliardo-Nirenberg inequality}
The Gagliardo-Nirenberg inequality \cite{Nirenberg1959} plays a crucial role in the discussion of prescribed mass problems. Specifically, for any $\Omega \subset \mathbb{R}^M$ and any $u \in H^1_0(\Omega)$,
\begin{equation}\label{eq2.5}
    \|u\|_{L^{p}(\Omega)}^{p} \leq C_{M, p}\|\nabla u\|_{L^2(\Omega)}^{\gamma_p p}\|u\|_{L^2(\Omega)}^{(1-\gamma_p)p},
\end{equation}
where $\gamma_p = \frac{M(p-2)}{2p} > 0$. Equality holds if and only if $\Omega = \mathbb{R}^M$ and $u$ is equal to $U_0$ up to translations and dilation, where $U_0$ is the unique positive solution of the problem
\begin{equation*}
    \begin{cases}
        -\Delta U_0 + U_0 = U_0^{p-1} & \text { in } \mathbb{R}^{M}, \\
        U_0>0 & \text { in } \mathbb{R}^{M}.
    \end{cases}
\end{equation*}
Moreover, $U_0(x)=U_0(|x|)$ has Morse index $1$ and satisfies the exponential decay estimate:
\begin{equation*}
    \left|D^\alpha U_0(x)\right| \leq C_1 \exp (-\sigma|x|), \quad x \in \mathbb{R}^{M},
\end{equation*}
for some $C_1, \sigma>0$ and all $|\alpha| \leq 2$.
Furthermore, for any $\lambda_0, \mu_0 > 0$, the unique positive solution of the problem
\begin{equation*}
    \begin{cases}
        -\Delta U+ \lambda_0 U = \mu_0 U^{p-1} & \text { in } \mathbb{R}^{M}, \\ 
        U>0 & \text { in } \mathbb{R}^{M},
    \end{cases}
\end{equation*}
is explicitly given by
\begin{equation*}
    U_{\lambda_0, \mu_0}(x)=\left(\frac{\mu_0}{\lambda_0}\right)^{\frac{1}{p-2}} U_0\left(\sqrt{\lambda_0} x\right).
\end{equation*}

\section{A global minimizer}

From now on, we will study the existence of positive solutions for the problem \eqref{eq:RGP}.
We first show that the energy functional $E$ is coercive on ${\S}_c$ for any $c > 0$ if $2 < q < 2^\sharp_M$ and for $c > 0$ small enough if $q = 2^\sharp_M$.
\begin{lemma}\label{prop3.1}
    Assume that the nonlinearity $f$ satisfies the assumptions ($\mathbf{A_1}$) and ($\mathbf{A_2}$), then $E(v)$ is coercive on ${\S}_c$ if one of the following conditions holds:
    \begin{enumerate}
        \item[(i)] $2 < q < 2^{\#}_M$ for any $c > 0$;
        \item[(ii)] $q = 2^{\#}_M$ for $c > 0$ sufficiently small.
    \end{enumerate}
    As a consequence, in both cases we have $\inf_{{\S}_c} E(v) > -\infty$.
\end{lemma}

\begin{proof}
    By ($\mathbf{A_1}$) and ($\mathbf{A_2}$), for any $\varepsilon > 0$, there exists some $C_\varepsilon > 0$ such that
    \begin{equation*}
        F(s) \leq \frac{\varepsilon}{2} |s|^2 + \frac{C_\varepsilon}{q} |s|^q, \quad s \in \mathbb{R}.
    \end{equation*}
    We will only give a proof the lemma for $\mu \geq 0$ and the proof for $\mu < 0$ is similar.
    Using the Gagliardo-Nirenberg inequality \eqref{eq2.5}, we have
    \begin{equation}\label{eq3.3}
        \begin{aligned}
            E(v) & =\frac{1}{2}\|\nabla v\|_{L^2(D)}^2-\int_D |x|^\mu F(v) \dif x \\
            & \geq \frac{1}{2}\|\nabla v\|_{L^2(D)}^2-\int_D |x|^\mu \left(\frac{\varepsilon}{2} |v|^2 + \frac{C_\varepsilon}{q} |v|^q \right) \dif x \\
            & \geq \frac{1}{2}\|\nabla v\|_{L^2(D)}^2 - \frac{\varepsilon}{2} c^2 - \frac{C_\varepsilon}{q} b^\mu \|v\|_{L^q(D)}^q \\
            & \geq \frac{1}{2}\|\nabla v\|_{L^2(D)}^2 - \frac{\varepsilon}{2} c^2 - \frac{C_\varepsilon}{q} b^\mu C_{M, q} \|\nabla v\|_{L^2(D)}^{\gamma_q q} \|v\|_{L^2(D)}^{(1-\gamma_q) q} \\
            & \geq \frac{1}{2}\|\nabla v\|_{L^2(D)}^2 - \frac{\varepsilon}{2} c^2 - \frac{C_\varepsilon}{q} b^\mu C_{M, q} \left(\frac{c^2}{a^\mu}\right)^{(1-\gamma_q) \frac{q}{2}} \|\nabla v\|_{L^2(D)}^{\gamma_q q} \\
            & \geq \frac{1}{2}\|\nabla v\|_{L^2(D)}^2 - C_1\|\nabla v\|_{L^2(D)}^{\gamma_q q} - C_2 \\
            & \geq \left(\frac{1}{2} \|\nabla v\|_{L^2(D)}^{2-\gamma_q q} - C_1\right) \|\nabla v\|_{L^2(D)}^{\gamma_q q} - C_2,
        \end{aligned}
    \end{equation}
    where $C_1$ depends on the parameters $M, q, a, b, c, \varepsilon, \mu$ and $C_2$ depends on $c, \varepsilon$. 
    
    Proof of (1): Since $2 < q < 2^\sharp_M$, $0 < \gamma_q q < 2$. So as $\|\nabla v\|_{L^2(D)} \rightarrow +\infty$, we have $\frac{1}{2} \|\nabla v\|_{L^2(D)}^{2-\gamma_q q} - C_1 \rightarrow +\infty$, so $E(v)$ is coercive in ${\S}_c$. 

    Proof of (2): Since $q = 2^\sharp_M$, $\gamma_q q = 2$. By \eqref{eq3.3}, if the mass $c > 0$ is sufficiently small, then the constant $C_1$ is also sufficiently small such that $\frac{1}{2} - C_1 > 0$. So in this case $E(v)$ is also coercive in ${\S}_c$.
\end{proof}

\begin{remark}
    From the proof of Lemma \ref{prop3.1}, if we fix the inner radius $a > 0$ and let the outer radius $b \rightarrow +\infty$, then the constant $C_1$ will tend to infinity when $\mu > 0$ or tend to $0$ when $\mu < 0$. 
    Therefore, this method cannot be used to show that $E(v)$ is bounded from below on ${\S}_c$ when $D$ is an exterior domain and $\mu > 0$.
    We will discuss the exterior problem in the final section.
\end{remark}

To prove the existence of a global minimizer, we also need the following compactness lemma. First, we give the definition of a Palais-Smale sequence for $E|_{{\S}_c}$.
\begin{definition}
    A sequence $\left\{v_n\right\} \subset {\S}_c$ is called a Palais-Smale sequence of $E|_{{\S}_c}$ if
    \begin{equation*}
        E\left(v_n\right) \text { is bounded and }  \left\|E'|_{{\S}_c}\left(v_n\right)\right\|_{H^{-1}} \rightarrow 0 \quad \text { as } n \rightarrow +\infty.
    \end{equation*}
\end{definition}

\begin{lemma}\label{lem4.2}
    Let $\left\{v_n\right\} \subset {\S}_c$ be a bounded Palais-Smale sequence of $E|_{\mathcal{S}_c}$. 
    Then there exists $v \in {\S}_c$ such that $v_n \rightarrow v$ strongly in $H_0^1(D)$ up to a subsequence.
\end{lemma}

\begin{proof}
    Since the sequence $\left\{v_n\right\} \subset {\S}_c$ is bounded in $H_0^1(D)$, we may deduce that up to a subsequence there is $v \in {\S}_c$ such that
    \begin{equation*}
        \begin{aligned}
         v_n \rightharpoonup v \quad & \text { in } H_0^1(D), \\
         v_n \rightarrow v \quad & \text { in } L^p(D), \quad \forall p \in \left[2,2^*\right), \\
         v_n(x) \rightarrow v(x) \quad & \text { for a.e. } x \in D .
        \end{aligned}
    \end{equation*}
    By the definition of the Palais-Smale sequence, there exists $\lambda_n \in \mathbb{R}$ such that
    \begin{equation}\label{eq4.12}
        \int_D \nabla v_n \nabla w \, \dif x+\lambda_n \int_D|x|^\mu v_n w \, \dif x- \int_D|x|^\mu f\left(v_n\right) w \, \dif x = o_n(1)
    \end{equation}
    for every $w \in H_0^1(D)$. Taking $w = v_n$ in \eqref{eq4.12} we have
    \begin{equation*}
        \lambda_n c^2= \int_D |x|^\mu f\left(v_n\right) v_n \dif x-\int_D\left|\nabla v_n\right|^2 \dif x + o_n(1).
    \end{equation*}
    Therefore the boundedness of $\left\{v_n\right\}$ in $H_0^1(D)$ implies the boundedness of $\lambda_n$, thus there exists $\bar{\lambda} \in \mathbb{R}$, up to subsequence, such that $\lambda_n \rightarrow \bar{\lambda}$. Replacing $w$ with $v_n-v$ in \eqref{eq4.12}, we have
    \begin{equation}\label{eq4.13}
        \int_D \nabla v_n \nabla\left(v_n-v\right) \dif x+\lambda_n \int_D |x|^\mu v_n\left(v_n-v\right) \dif x - \int_D|x|^\mu f\left(v_n\right)\left(v_n-v\right) \dif x=o_n(1).
    \end{equation}
    By the weak convergence of $v_n$ in $H_0^1(D)$, we have
    \begin{equation}\label{eq4.14}
        \int_D \nabla v \nabla (v_n-v)+\bar{\lambda} \int_D |x|^\mu v\left(v_n-v\right) \dif x = o_n(1)
    \end{equation}
    Combining \eqref{eq4.13}, \eqref{eq4.14} and the strong convergence of $v_n$ in $L^p(D)$ and the convergence of $\lambda_n$, we obtain
    \begin{equation*}
        \begin{aligned}
        o_n(1)  = &\int_D\left|\nabla(v_n-v)\right|^2 \dif x+\left(\lambda_n-\bar{\lambda}\right) \int_D|x|^\mu v\left(v_n-v\right) \dif x \\
        & + \lambda_n \int_D |x|^\mu \left(v_n-v\right)^2 \dif x -\tau \int_D |x|^\mu f\left(u_n\right)\left(v_n-v\right) \dif x \\
         = &\int_D \left|\nabla\left(v_n-v\right)\right|^2 \dif x+o_n(1).
        \end{aligned}
    \end{equation*}
    Thus, $v_n \rightarrow v$ strongly in $H_0^1(D)$.
\end{proof}

\begin{proof}[Proof of Theorem \ref{thm1.2} (i) and (ii)]
    (i) For any $c > 0$, by Proposition \ref{prop3.1}, we have $\inf\limits_{{\S}_c} E(v) > -\infty$. 
    Therefore, we can find a minimizing sequence $\{v_n\} \subset {\S}_c$ associated to $\inf_{{\S}_c} E(v)$. 
    From the coerciveness of $E(v)$ on ${\S}_c$, we can deduce that $\{v_n\}$ is bounded in $H_0^1(D)$.
     By condition ($\mathbf{A_1}$) we may substitute $\{v_n\}$ with $\{|v_n|\}$ so we can assume that $\{v_n\}$ is nonnegative. 
     Using the Ekeland variational principle (see \cite[Theorem 2.4]{Willem1996}), we can deduce that, by choosing another minimizing sequence if necessary, $\{v_n\}$ is a bounded nonnegative Palais-Smale sequence in $H_0^1(D)$. 
     So using Lemma \ref{lem4.2}, we can find a global minimizer $v_c \geq 0$ that attains the minimum: $E(v_c) = \inf\limits_{{\S}_c} E(v)$. 
     By the strong maximum principle, $v_c > 0$ in $D$. 

    (ii) The proof is similar to (i) if we choose $c_1 > 0$ sufficiently small such that $E(v)$ is coercive on ${\S}_c$ for any $0 < c < c_1$.
\end{proof}

\section{A local minimizer and a mountain pass solution}

\subsection{Existence of a local minimizer}
We now assume that $2^\sharp_M < q < \frac{2M}{M-2}$. 
In this case, the functional $E(v)$ is unbounded from below on ${\S}_c$ so we cannot hope to find a global minimizer. 
However, it is not hard to find a local minimizer when the mass $c$ is sufficiently small.

For future convenience, we introduce the following auxiliary problem
\begin{equation}\label{eq:AP}
    \begin{cases}
        -\Delta v+ \lambda |x|^\mu v= \tau |x|^\mu f(v) & \text { in } D, \\ 
        v=0 & \text { on } \partial D, \\
        \int_D |x|^\mu v^2 \, \dif x = c^2,
    \end{cases}
\end{equation}
where $\tau \in [\frac{1}{2},1]$. Solutions to this problem can be found as critical points of the energy functional
\begin{equation*}
    E_\tau(v) = \frac{1}{2} \int_D |\nabla v|^2 \, \dif x - \tau \int_D |x|^\mu F(v) \, \dif x
\end{equation*}
on ${\S}_c$. In particular, when $\tau =1$ we are back to the problem \eqref{eq:RGP}.

In order to investigate the geometric structure of $E_\tau$, we define the sets 
\begin{equation*}
    \mathcal{B}_\rho=\left\{u \in {\S}_c: \int_D |\nabla u|^2 \dif x<\rho \right\}, \quad \mathcal{U}_\rho=\left\{u \in {\S}_c: \int_D |\nabla u|^2 \dif x=\rho \right\} .
\end{equation*}

Suppose that $\varphi_1 \in \S_c$ is the first eigenfunction that satisfies the problem
\begin{equation*}
    \begin{cases}
        -\Delta \varphi_1=\lambda_1 \varphi_1 & \text { in } D,\\ 
        \varphi_1=0 & \text { on } \partial D.
    \end{cases}
\end{equation*}
where $\lambda_1$ is the first eigenvalue of $-\Delta$ in $H_0^1(D)$.
By a direct calculation, $\varphi_1 \in \mathcal{B}_\rho$ if 
\begin{equation*}
    \begin{array}{ll}
        \rho \geq \lambda_1 b^{-\mu} c^2 & \text{ when } \mu >0, \\
        \rho \geq \lambda_1 c^2 & \text{ when } \mu =0, \\
        \rho \geq \lambda_1 a^{-\mu} c^2 & \text{ when } \mu <0,
    \end{array}
\end{equation*}  
in which case $\mathcal{B}_\rho$ is nonempty.
Now, we will show that the functional $E_\tau$ has a local minimizer when the mass $c$ is sufficiently small.
\begin{lemma}\label{prop4.1}
    If $2^\sharp_M < q < \frac{2M}{M-2}$, then there is a $c_2 > 0$ such that when $0 < c < c_2$, for some $\rho^* > 0$ we have
    \begin{equation}\label{eq4.2}
        m_\tau := \inf _{\mathcal{B}_{\rho^*}} E_\tau(v)< \inf _{\mathcal{U}_{\rho^*}} E_\tau(v).
    \end{equation}
\end{lemma}

\begin{proof}
    By ($\mathbf{A_1}$) and ($\mathbf{A_2}$), for any $\varepsilon > 0$, there exists some $C_\varepsilon > 0$ such that
    \begin{equation*}
        F(s) \leq \frac{\varepsilon}{2} |s|^2 + \frac{C_\varepsilon}{q} |s|^q, \quad s \in \mathbb{R}.
    \end{equation*}
    We will only proof the lemma for $\mu \geq 0$ and the proof for $\mu < 0$ is similar.
    For $v \in \mathcal{U}_\rho$, we have
    \begin{equation*}
        \begin{aligned}
            E_\tau (v) & =\frac{1}{2}\|\nabla v\|_{L^2(D)}^2-\tau \int_D |x|^\mu F(u) \dif x \\
            & \geq \frac{1}{2}\|\nabla v\|_{L^2(D)}^2- \tau \int_D |x|^\mu \left(\frac{\varepsilon}{2} |v|^2 + \frac{C_\varepsilon}{q} |v|^q \right) \, \dif x \\
            & = \frac{1}{2}\|\nabla v\|_{L^2(D)}^2 - \tau \left(\frac{\varepsilon}{2} c^2 + \frac{C_\varepsilon}{q} b^\mu \|v\|_{L^q(D)}^q \right) \\
            & \geq \frac{1}{2}\rho - \tau \left(\frac{\varepsilon}{2} c^2 + \frac{C_\varepsilon}{q} b^\mu C_{M, q} \|\nabla v\|_{L^2(D)}^{\gamma_q q} \|v\|_{L^2(D)}^{(1-\gamma_q) q} \right) \\
            & \geq \frac{1}{2}\rho - \tau \left(\frac{\varepsilon}{2} c^2 + \frac{C_\varepsilon}{q} b^\mu C_{M, q} \left(\frac{c^2}{a^\mu}\right)^{(1-\gamma_q) \frac{q}{2}} \rho^{\frac{\gamma_q q}{2}} \right) \\
            & = \frac{1}{2}\rho - C_1 \tau c^2 - C_2 \tau \rho^{\frac{\gamma_q q}{2}} c^{(1-\gamma_q) q} \\
            & := f(\rho),
        \end{aligned}
    \end{equation*}
    where $C_1 = \frac{\varepsilon}{2}$, $C_2$ depends on $M, q, a, b, \varepsilon, \mu$. Now we have
    \begin{equation}\label{eq4.6}
        f'(\rho) = \frac{1}{2} - C_2 \tau \frac{\gamma_q q}{2} \rho^{\frac{\gamma_q q}{2}-1} c^{(1-\gamma_q) q}.
    \end{equation}
    Because of $\frac{\gamma_q q}{2} > 1$, we may deduce that there is a unique $\rho^*$ such that $f'(\rho^*) = 0$, i.e.
    \begin{equation*}
        \begin{aligned}
            & f(\rho^*) = \max f(\rho) 
             = C_2 \tau \left(\frac{\gamma_q q}{2}-1\right) (\rho^*)^{\frac{\gamma_q q}{2}} c^{(1-\gamma_q) q} - C_1 \tau c^2.
        \end{aligned}
    \end{equation*}
    So we have 
    \begin{equation*}
        \begin{aligned}
            \inf _{\mathcal{U}_{\rho^*}} E_\tau(v) \geq C_2 \tau \left(\frac{\gamma_q q}{2}-1\right) (\rho^*)^{\frac{\gamma_q q}{2}} c^{(1-\gamma_q) q} - C_1 \tau c^2.
        \end{aligned}
    \end{equation*}
    On the other hand, we have
    \begin{equation*}
        \inf_{\mathcal{B}_{\rho^*}} E_\tau(v) \leq E_\tau(\varphi_1) \leq \frac{1}{2} \lambda_1 \tau a^{-\mu} c^2.
    \end{equation*}
    As a consequence, \eqref{eq:AP} holds if we have
    \begin{equation*}
            C_2 \tau \left(\frac{\gamma_q q}{2}-1\right) (\rho^*)^{\frac{\gamma_q q}{2}} c^{(1-\gamma_q) q} - C_1 \tau c^2 \\
            > \frac{1}{2} \lambda_1 \tau a^{-\mu} c^2,
    \end{equation*}
    that is,
    \begin{equation}\label{eq4.10}
        C_2 \left(\frac{\gamma_q q}{2}-1\right) (\rho^*)^{\frac{\gamma_q q}{2}} c^{(1-\gamma_q) q-2} > C_1 + \frac{1}{2} \lambda_1 a^{-\mu},
    \end{equation}
    and the condition
    \begin{equation}\label{eq4.11}
        \rho^* \geq \lambda_1 b^{-\mu} c^2
    \end{equation}
    also needs to be satisfied to guarantee $\mathcal{B}_{\rho^*}$ is non-empty.

    Notice that \eqref{eq4.6} shows that $\rho^*$ increases as the mass $c$ decreases. So for condition \eqref{eq4.11}, as the mass $c$ decreases the left-hand side increases and the right-hand side decreases. So there is a $c_1^*>0$ such that \eqref{eq4.11} holds when $0 < c < c_1^*$.
    
    For condition \eqref{eq4.10}, from $M \geq 3$ we have $ (1-\gamma_q) q - 2 < 0$. So for $c$ decreases the left-hand side increases. So there is a $c_2^*>0$ such that \eqref{eq4.10} holds when $0 < c < c_2^*$. As a result, if we choose $c_2 = \min\{c_1^*, c_2^*\}$, \eqref{eq4.2} holds. The proof for the case $\mu < 0$ is similar.
\end{proof}

\begin{proof}[Proof of Theorem \ref{thm1.2} (iii): Existence of a local minimizer]
    By Proposition \ref{prop4.1}, for any $c \in \left(0, c_2\right)$ there exists a minimizing sequence $\{v_n\} \subset \mathcal{B}_{\bar{\rho}}$ associated to $m = m_1$, by ($\mathbf{A_1}$) we can assume that $v_n \geq 0$ for every $n$. From the Ekeland variational principle, by choosing another minimizing sequence if necessary, $\{v_n\} \subset \mathcal{B}_\rho$ is a bounded nonnegative Palais-Smale sequence for $E_1 = E$ at level $m$. Using Lemma \ref{lem4.2}, we can find a minimizer $\tilde{v}_c \geq 0$ such that $E(\tilde{v}_c) = m$. By the strong maximum principle, $\tilde{v}_c > 0$ in $D$.
\end{proof}

\subsection{The mountain pass geometry}

With the local minimizer above, we can also show the mountain pass geometry of the functional $E_\tau$ on ${\S}_c$ uniformly with respect to $\tau \in [\frac{1}{2},1]$.
\begin{proposition}\label{prop4.3}
    For each $c \in (0, c_2)$, there exist $w_1, w_2 \in {\S}_c$ independent of $\tau$, such that
    \begin{equation*}
    \tilde{m}_\tau:=\inf _{\gamma \in \Gamma} \max _{t \in [0,1]} E_\tau(\gamma(t))>\max \left\{E_\tau\left(w_1\right), E_\tau\left(w_2\right)\right\}, \forall \tau \in\left[\frac{1}{2}, 1\right],
    \end{equation*}
    where $\Gamma=\left\{\gamma \in C\left([0,1], {\S}_c\right) \mid \gamma(0)=w_1, \gamma(1)=w_2\right\}$.
\end{proposition}

\begin{proof}
    First we notice that from \eqref{eq4.6}, the value $\rho^*$ decreases as $\tau \in\left[\frac{1}{2}, 1\right]$ decreases. 
    So we may assume that $\rho^*$ is the one corresponding to $\tau = \frac{1}{2}$. Now \eqref{eq4.10} and \eqref{eq4.11} hold for any $\tau \in\left[\frac{1}{2}, 1\right]$ when $0 < c < c_2$. 

    Now we can choose $w_1 = \varphi_1$, which is independent of $\tau$. To obtain $w_2$ we choose $\phi \in C_0^{\infty}\left(B_1(0)\right)$ with $\phi>0$ in $B_1(0)$ and $\int_{B_1(0)} \phi^2 \, \dif x=1$, and $x_1 \in D$. For $k \in \mathbb{N}$, we may define
    \begin{equation*}
        v_k(x)= \chi k^{\frac{M}{2}} \phi\left(k\left(x-x_1\right)\right), x \in D,
    \end{equation*} 
    where $\chi>0$ is chosen such that $\int_D |x|^\mu v_k^2 \dif x = c^2$. By definition, $\operatorname{supp}\left(v_k\right) \subset B_{1/k}\left(x_1\right) \subset D$ for $k$ sufficiently large. Thus, for $k$ sufficiently large, $v_k \in {\S}_c$.
    Now
    \begin{equation*}
        c^2 = \int_D |x|^\mu v_k^2 \dif x \sim \chi^2 \int_{B_1(0)} \phi^2 \dif x = \chi^2.
    \end{equation*} 
    Thus we have $\chi \sim c$. By a straightforward calculation, it follows that as $k \rightarrow +\infty$,
    \begin{equation*}
        \left\|v_k\right\|_{H_0^1(D)} \sim c \left(k^2 \int_{B_1(0)}|\nabla \phi|^2 \dif x+1\right)^{\frac{1}{2}} \rightarrow+\infty.
    \end{equation*}
    By ($\mathbf{A_2}$) there exists some $R_0 > 0$ such that
    \begin{equation*}
        F(s) \geq \frac{a_0}{2q} |s|^q, \quad \forall |s| \geq R_0.
    \end{equation*} 
    If we denote $D_k = \left\{x \in D: |v_k(x)| \geq R_0\right\}$, then for $\mu \geq 0$ we have
    \begin{equation*}
        \begin{aligned}
            E_\tau\left(v_k\right) & \leq E_{{1}/{2}}\left(v_k\right) \\
            & = \frac{1}{2} \int_D |\nabla v_k|^2 \dif x - \frac{1}{2} \int_D |x|^\mu F\left(v_k\right) \dif x \\
            & \leq \frac{1}{2} \int_D |\nabla v_k|^2 \dif x - \frac{a_0}{4q} \int_{D_k} |x|^\mu |v_k|^q \dif x \\
            & \leq \frac{1}{2} \int_D |\nabla v_k|^2 \dif x - \frac{a_0}{4q} a^\mu \int_{D_k} |v_k|^q \dif x \\
            & \leq \frac{\chi^2 k^2}{2} \int_{B_1(0)} |\nabla \phi|^2 \dif x + \frac{a_0 R_0^q}{4q} |D| - \frac{a_0 \chi^q k^{\frac{M(q-2)}{2}}}{4q} \int_{B_1(0)} |\phi|^q \dif x \rightarrow -\infty
        \end{aligned}
    \end{equation*}
    as $k \rightarrow +\infty$ since $M (q-2)/2 > 2$ when $q > 2^\sharp_M$. The case $\mu < 0$ is similar by interchanging $a$ and $b$. So we can take a $k_0>0$ sufficiently large and independent of $\tau$, and let $w_2=u_{k_0}$. Then we have
    \begin{equation}\label{eq4.16}
        \int_{\Omega}\left|\nabla w_2\right|^2 \dif x > 2\rho^* \text { and } E_\tau\left(w_2\right) \leq E_{\frac{1}{2}}\left(w_2\right)< m_{\frac{1}{2}}, \forall \tau \in\left[\frac{1}{2}, 1\right] .
    \end{equation}
    Now we let $\Gamma$ and $m_\tau$ be defined as in the statement of the proposition for our choice of $w_1$ and $w_2$, then $\Gamma \neq \emptyset$ holds since
    \begin{equation*}
        \gamma_0(t)=\frac{c}{\left\|(1-t) w_1+t w_2\right\|_{L^2(\Omega)}}\left[(1-t) w_1+t w_2\right], \quad t \in[0,1]
    \end{equation*}
    belongs to $\Gamma$. In particular, since $w_1 \in \mathcal{B}_{\rho^*}$, \eqref{eq4.16} infers that for any $\gamma \in \Gamma$, $\gamma([0,1])$ intersects $\mathcal{U}_{\rho^*}$. Therefore, by definition
    \begin{equation*}
        \tilde{m}_\tau:=\inf _{\gamma \in \Gamma} \max _{t \in [0,1]} E_\tau(\gamma(t)) \geq \inf_{\mathcal{U}_{\rho^*}} E_\tau>\max \left\{E_\tau\left(w_1\right), E_\tau\left(w_2\right)\right\}, \forall \tau \in\left[\frac{1}{2}, 1\right],
    \end{equation*}
    which completes the proof.
\end{proof}

\subsection{The abstract framework of monotonicity trick}

To establish the existence of normalized solutions of mountain pass type for problem \eqref{eq:AP}, we need to use a variant of the monotonicity trick in \cite{Borthwick2024}. We first introduce some notations and definitions.

Let $(E,\langle\cdot, \cdot\rangle)$ and $(H,(\cdot, \cdot))$ be two infinite-dimensional Hilbert spaces such that $E \hookrightarrow H \hookrightarrow E^{\prime}$ with continuous injections. For simplicity, we suppose that $E \hookrightarrow H$ has norm at most 1 and identify $E$ with its image in $H$. We define
\begin{equation*}
    \begin{cases}
        \|u\|^2=\langle u, u\rangle, \\
        |u|^2=(u, u),
    \end{cases}
\end{equation*}
for each $u \in E$, and
\begin{equation*}
    \mathcal{S}_c=\left\{u \in E| \, |u|^2=c\right\}, \forall c>0 .
\end{equation*}
Obviously, $\mathcal{S}_c$ is a submanifold of $E$ of codimension 1 and its tangent space at a given point $u \in \mathcal{S}_c$ is given by
\begin{equation*}
T_u \mathcal{S}_c=\{v \in E \mid(u, v)=0\}.
\end{equation*}
Denote by $\|\cdot\|_*$ and $\|\cdot\|_{* *}$, respectively, the operator norm of $\mathcal{L}(E, \mathbb{R})$ and $\mathcal{L}(E, \mathcal{L}(E, \mathbb{R}))$.  
\begin{definition}[{\cite[Definition 1.1, Definition 1.3, Definition 1.4]{Borthwick2024}}]
    Let $\phi: E \rightarrow \mathbb{R}$ be a $C^2$-functional on $E$.
    \begin{enumerate}
        \item[(i)] We say that $\phi^{\prime}$ and $\phi^{\prime \prime}$ are $\alpha$-Hölder continuous on bounded sets for some $\alpha \in(0,1]$ if for any $R > 0$ one can find $M=M(R)>0$ such that, for any $u_1, u_2 \in B_R(0)$,
                \begin{equation*}
                    \left\|\phi^{\prime}-\phi^{\prime}\right\|_* \leq M\left\|u_1-u_2\right\|^\alpha,\left\|\phi^{\prime \prime}-\phi^ {\prime \prime}\right\|_{* *} \leq M\left\|u_1-u_2\right\|^\alpha .
                \end{equation*}
        \item[(ii)] For any $u \in E$, we define the continuous bilinear map
                \begin{equation*}
                    D^2 \phi(u)=\phi^{\prime \prime}(u)-\frac{\phi^{\prime}(u) \cdot u}{|u|^2}(\cdot, \cdot) .
                \end{equation*}
        \item[(iii)] For any $u \in \mathcal{S}_c$ and $\theta>0$, we define the approximate Morse index by
                \begin{equation*}
                    \begin{aligned}
                    \tilde{m}_\theta(u)= & \sup \left\{\operatorname{dim} L \mid L \text { is a subspace of } T_u \S_c \text { such that } D^2 \phi(u)[\varphi, \varphi]<-\theta\|\varphi\|^2, \forall \varphi \in L \backslash\{0\}\right\} .
                    \end{aligned}
                \end{equation*}
                If $u$ is a critical point for the constrained functional $\left.\phi\right|_{\S_c}$ and $\theta=0$, we denote $\tilde{m}(u)$ is the Morse index of $u$ as constrained critical point.
    \end{enumerate}
\end{definition}

Now we can state the abstract result of the monotonicity trick.

\begin{lemma}[{\cite[Theorem 1.5]{Borthwick2024}}]\label{lem4.5}
    Let $I \subset(0,+\infty)$ be an interval and consider a family of $C^2$ functional $\Phi_\rho: E \rightarrow \mathbb{R}$ of the form
    \begin{equation*}
        \Phi_\rho(u)=A(u)-\rho B(u), \quad \rho \in I,
    \end{equation*}
    where $B(u) \geq 0$ for every $u \in E$, and
    \begin{equation*}
        \text { either } A(u) \rightarrow+\infty \text { or } B(u) \rightarrow+\infty \text { as } u \in E \text { and }\|u\| \rightarrow+\infty \text {. }
    \end{equation*}

    Suppose moreover that $\Phi_\rho^{\prime}$ and $\Phi_\rho^{\prime \prime}$ are $\alpha$-Hölder continuous on bounded sets for some $\alpha \in(0,1]$. Finally, suppose that there exist $w_1, w_2 \in \mathcal{S}_c$ (independent of $\rho$) such that, setting
    \begin{equation*}
        \Gamma=\left\{\gamma \in C\left([0,1], \mathcal{S}_c\right) \mid \gamma(0)=w_1, \gamma(1)=w_2\right\},
    \end{equation*}
    we have
    \begin{equation*}
        c_\rho=\inf _{\gamma \in \Gamma} \max _{t \in[0,1]} \Phi_\rho(\gamma(t))>\max \left\{\Phi_\rho\left(w_1\right), \Phi_\rho\left(w_2\right)\right\}, \quad \rho \in I .
    \end{equation*}

    Then, for almost every $\rho \in I$, there exist sequences $\left\{u_n\right\} \subset \S_c$ and $\zeta_n \rightarrow 0^{+}$ such that, as $n \rightarrow+\infty$,
    \begin{enumerate}
        \item[(i)] $\Phi_\rho\left(u_n\right) \rightarrow c_\rho$;
        \item[(ii)] $\left\|\Phi_\rho^{\prime} \mid_{\mathcal{S}_c}(u_n) \right\|_* \rightarrow 0$;
        \item[(iii)] $\left\{u_n\right\}$ is bounded in $E$;
        \item[(iv)] $\tilde{m}_{\zeta_n}\left(u_n\right) \leq 1$.
    \end{enumerate}
\end{lemma}

Now we can use the abstract framework above to attain a mountain pass critical point of $E_{\tau}$ for almost every $\tau \in [\frac{1}{2},1]$. For our propose, in the following we denote $E = H_0^1(D)$ and $H = L^2(D)$, equipped with another inner product
    \begin{equation*}
            (u, v)_H=\int_D|x|^\mu u v \, \dif x 
    \end{equation*}
    and norm
    \begin{equation*}
        \|v\|_H=\left(\int_D|x|^\mu v^2 \, \dif x\right)^{\frac{1}{2}}.
    \end{equation*}
    Due to the boundedness of $D$, the norm $\|\cdot\|_H$ is equivalent to the standard norm of $L^2(D)$ for any $\mu \in \mathbb{R}$.
\begin{lemma}\label{lem4.6}
    Let $0< c < c_2$. For almost every $\tau \in \left[\frac{1}{2}, 1\right]$, there exists $\left(v_\tau, \lambda_\tau\right) \in \mathcal{S}_c \times \mathbb{R}$ which solves \eqref{eq:AP}. Moreover, $m\left(v_\tau\right) \leq 2$.
\end{lemma}

\begin{proof}
    We can utilize Lemma \ref{lem4.5} to the family of functionals $E_\tau$, with $E=H_0^1(D)$, $H=L^2(D)$ equipped with the inner product above, $c_\rho = \tilde{m}_\tau$ and $\Gamma$ defined in Proposition \ref{prop4.3}. If we set 
    \begin{equation*}
        A(v) = \frac{1}{2} \int_D |\nabla v|^2  \dif x \quad \text{ and } \quad B(v) = \tau \int_D |x|^\mu F(v) \, \dif x,
    \end{equation*}
    then $A(v)$ and $B(v)$ satisfy the assumptions in Lemma \ref{lem4.5}. Since $E_\tau$ is $C^2$ on ${\S}_c$, $E_\tau^\prime$ and $E_\tau^{\prime \prime}$ are local Hölder continuous on ${\S}_c$. Thus all the assumptions of Lemma \ref{lem4.5} are satisfied.

    Therefore, for almost $\tau \in \left[\frac{1}{2}, 1\right]$, there exists a bounded Palais-Smale sequence $\left\{v_{n, \tau}\right\} \subset H_0^1(D)$ for $E_\tau$ constrained on ${\S}_c$ at the level $\tilde{m}_\tau$ and a sequence $\left\{\zeta_n\right\} \subset \mathbb{R}^{+}$ with $\zeta_n \rightarrow 0^{+}$ satisfying $\tilde{m}_{\zeta_n}\left(v_{n, \tau}\right) \leq 1$. For simplicity, we still denote $v_{n, \tau}$ by $v_n$. Since the map $v \mapsto |v|$ is continuous, and $E_\tau(v)=E_\tau(|v|)$, we may assume that $v_n \geq 0$. Using Lemma \ref{lem4.2} for the functional $E_\tau$, we derive that there exists $0 \leq v_\tau \in \mathcal{S}_c$ such that $v_n \rightarrow v_\tau$ in $H_0^1(D)$, which implies that, for some $\lambda_\tau \in \mathbb{R},\left(v_\tau, \lambda_\tau\right)$ satisfies the problem
    \begin{equation*}
        \begin{cases}
            -\Delta v_\tau+ \lambda_\tau |x|^\mu v_\tau= \tau |x|^\mu f(v_\tau) & \text { in } D, \\ 
            v_\tau =0 & \text { on } \partial D, \\
            \int_D |x|^\mu v_\tau^2 \, \dif x = c^2.
        \end{cases}
    \end{equation*}
    By the strong maximum principle it follows that $v_\tau > 0$ in $D$.

    Now we will show that $m\left(v_\tau\right) \leq 2$. First, we deduce that $\tilde{m}(v_\tau) \leq 1$. Suppose by contradiction that $\tilde{m}(v_\tau) \geq 2$. Then there exists a subspace $W \subset T_0 {\S}_c$ with $\operatorname{dim} W=2$ such that
    \begin{equation*}
        D^2 E_\tau(v_\tau)[\varphi, \varphi]<0, \quad \forall \varphi \in W \backslash\{0\}.
    \end{equation*}
    Since $W$ is finite-dimensional, by the continuity of $D^2 E_\tau$ and we have that there exists some $\delta>0$ such that
    \begin{equation*}
        D^2 E_\tau(v)[\varphi, \varphi] < -\delta \|\varphi\|^2, \quad \forall \varphi \in W \backslash\{0\}.
    \end{equation*}
    Since $D^2 E_\tau$ is $\alpha$-Hölder continuous on bounded sets, we have that there exists some $\bar{\delta} > 0$ such that if $\|v - v_\tau\| < \bar{\delta}$, then
    \begin{equation*}
        D^2 E_\tau(v)[\varphi, \varphi] < -\frac{\delta}{2} \|\varphi\|^2, \quad \forall \varphi \in W \backslash\{0\}.
    \end{equation*}
    Since $v_n \rightarrow v_\tau$ in $H_0^1(D)$, there exists some $n_0 \in \mathbb{N}$ such that for all $n \geq n_0$, $\|v_n - v_\tau\| < \bar{\delta}$, which implies that
    \begin{equation*}
        D^2 E_\tau(v_n)[\varphi, \varphi] < -\frac{\delta}{2} \|\varphi\|^2, \quad \forall \varphi \in W \backslash\{0\}.
    \end{equation*}
    Therefore, $\tilde{m}_{\zeta_n}(v_n) \geq 2$ for all $n \geq n_0$, contradicting the fact that $\tilde{m}_{\zeta_n}(v_n) \leq 1$. Thus, we have $\tilde{m}(v_\tau) \leq 1$. Since ${\S}_c$ is of codimension $1$ in $H_0^1(D)$, by the definition of Morse index, we have $m(v_\tau) \leq \tilde{m}(v_\tau) + 1 \leq 2$.
\end{proof}

\subsection{The blow-up analysis}

Using Lemma \ref{lem4.6}, we can establish that there exists a sequence $\tau_n \rightarrow 1^{-}$ and a corresponding sequence of mountain pass critical points $\left\{v_{\tau_n}\right\}$ of $E_{\tau_n}$ on ${\S}_c$ at the level $\tilde{m}_{\tau_n}$ with a Morse index $m\left(v_{\tau_n}\right) \leq 2$. In order to attain a mountain pass critical point of the original problem, we need to demonstrate the convergence of $\left\{v_{\tau_n}\right\}$ to a critical point $v$ of $E|_{{\S}_c}$. To achieve this, a crucial step is to establishing the boundedness of $\left\{v_{\tau_n}\right\}$ in $H_0^1(D)$, which can be obtained if the sequence of Lagrange multiplier $\left\{\lambda_{\tau_n}\right\}$ is bounded. So in this subsection, we use a blow-up analysis method originated from \cite{Esposito2011} (see also \cite{Pierotti2017, Chang2023}) to achieve this goal.

For notational convenience, we denote $v_n:=v_{\tau_n}$, $\tilde{m}_n:=\tilde{m}_{\tau_n}$, $E_n:=E_{\tau_n}$. We will analyse the behavior of a sequence of positive solutions $\{v_n\} \subset {\S}_c$ of the following problem
\begin{equation}\label{eq4.17}
    \begin{cases}
        -\Delta v_n+ \lambda_n |x|^\mu v_n= \tau_n |x|^\mu f(v_n) & \text { in } D, \\ 
        v_n = 0 & \text { on } \partial D, \\
    \end{cases}
\end{equation}
as $\lambda_n \rightarrow +\infty$, where $\tau_n \rightarrow 1^{-}$, $m\left(v_n\right) \leq 2$ for all $n \in \mathbb{N}$.

\textbf{Step 1:} If $\lambda_n \rightarrow +\infty$, then $\{v_n\}$ blows up along any sequence of local maximum points.

\begin{lemma}\label{lem4.7}
    Assume that ($\mathbf{A_1}$) and ($\mathbf{A_2}$) hold. Let $\left\{v_n\right\} \subset H_0^1(D)$ be positive solutions to \eqref{eq4.17} with $\lambda_n \rightarrow +\infty$. Let $P_n$ be a local maximum point for $v_n$. Then there exists some $C > 0$ independent of $n$ such that
    \begin{equation}\label{eq4.18}
        v_n\left(P_n\right) \geq C \lambda_n^{\frac{1}{q-2}}.
    \end{equation}
\end{lemma}

\begin{proof}
    By elliptic regularity arguments, we can deduce that $v_n \in C^2(D)$. From ($\mathbf{A_1}$) and ($\mathbf{A_2}$), there exists some $C > 0$ such that
    \begin{equation*}
        f(s) \leq C (|s| + |s|^{q-1}), \quad \forall s \in \mathbb{R}.
    \end{equation*}
    Using the fact that $P_n$ is a local maximum point for $v_n$, we have $\Delta v_n\left(P_n\right) \leq 0$. So by \eqref{eq4.17}, we have
    \begin{equation*}
        \lambda_n\left|P_n\right|^\mu v_n\left(P_n\right) \leq \tau_n\left|P_n\right|^\mu f\left(v_n\left(P_n\right)\right),
    \end{equation*}
    i.e. 
    \begin{equation*}
        \lambda_n v_n\left(P_n\right) \leq f\left(v_n\left(P_n\right)\right) \leq C(v_n\left(P_n\right) + \left(v_n\left(P_n\right)\right)^{q-1}),
    \end{equation*}
    Assume that $v_n\left(P_n\right) < 1$, then we have $\lambda_n v_n\left(P_n\right) \leq C$ for all $n$. Therefore $v_n\left(P_n\right) \rightarrow 0$. This implies that $v_n \rightarrow 0$ uniformly on $D$, contradicting the fact that $v_n \in {\S}_c$. So $v_n\left(P_n\right) \geq 1$. Therefore,
    \begin{equation*}
        \lambda_n v_n\left(P_n\right) \leq C \left(v_n\left(P_n\right)\right)^{q-1},
    \end{equation*}
    which implies \eqref{eq4.18}.
\end{proof}

\textbf{Step 2:} We deduce a precise behavior of the sequence $\left\{v_n\right\}$ near the local maximum points $P_n$ of $v_n$ as $\lambda_n \rightarrow+\infty$.

\begin{lemma}\label{lem4.8}
    Let $\lambda_n \rightarrow +\infty$ as $n \rightarrow+\infty$ and $P_n \in D$ be such that, for some $R_n \rightarrow+\infty$,
    \begin{equation*}
        v_n\left(P_n\right)=\max_{D \cap B_{R_n \tilde{\varepsilon}_n}\left(P_n\right)} v_n, \quad \text { where } \tilde{\varepsilon}_n=\left(v_n\left(P_n\right)\right)^{-\frac{q-2}{2}} \rightarrow 0.
    \end{equation*}
    If we set the rescaled function
    \begin{equation*}
        V_n(y)=\varepsilon_n^{\frac{2}{q-2}} v_n\left(\varepsilon_n y+ P_n\right) \quad \text { for } y \in D_n=\frac{D-P_n}{\varepsilon_n }, \quad \text { with } \varepsilon_n=\lambda_n^{-\frac{1}{2}}.
    \end{equation*}

    Then, up to a subsequence, we have
    \begin{enumerate}
        \item[(i)] $\varepsilon_n(\operatorname{dist}\left(P_n, \partial D\right))^{-1} \rightarrow 0$ as $n \rightarrow+\infty$.
        \item[(ii)] $v_n\left(P_n\right) = \max\limits_{D \cap B_{R_n \varepsilon_n}\left(P_n\right)}v_n$ for some $R_n \rightarrow+\infty$.
        \item[(iii)] $V_n \rightarrow V$ in $C_{\mathrm{loc}}^1\left(\mathbb{R}^{M}\right)$ as $n \rightarrow+\infty$, where $V$ solves
                    \begin{equation}\label{eq4.19}
                        \begin{cases}
                            -\Delta V + |P|^{\mu} V = a_0 |P|^{\mu} V^{q-1} & \text { in } \, \mathbb{R}^{M}, \\ 
                            0 < V \leq V(0) & \text { in } \, \mathbb{R}^{M}, \\ 
                            V \rightarrow 0 & \text { as } \, |x| \rightarrow +\infty.
                        \end{cases}
                    \end{equation}
        \item[(iv)]  There exists $\phi_n \in C_0^{\infty}(D)$ with $\operatorname{supp} \phi_n \subset B_{R \varepsilon_n }\left(P_n\right), R>0$, such that for $n$ large
                    \begin{equation}\label{eq4.20}
                        \int_D\left(\left|\nabla \phi_{n}\right|^2+\lambda_n|x|^\mu \phi_{n}^2-\tau_n|x|^\mu f^{\prime}\left(v_n\right) \phi_{n}^2 \right)\dif x<0 .
                    \end{equation}
        \item[(v)] For all $R>0$ and $q \geq 1$ there holds
                    \begin{equation}\label{eq4.21}
                        \lim _{n \rightarrow+\infty} \lambda_n^{\frac{N}{2}-\frac{q}{q-2}} \int_{B_{R \varepsilon_n }\left(P_n\right)} |x|^{\mu} v_n^q \dif x=\int_{B_R(0)} |P|^{\mu} V^q \dif x.
                    \end{equation}
    \end{enumerate}
\end{lemma}

\begin{proof}
    We first define a rescaled function
    \begin{equation*}
        \tilde{V}_n(y):=\tilde{\varepsilon}_n^{\frac{2}{q-2}} v_n\left(\tilde{\varepsilon}_n y + P_n\right) \text{ for } y \in \tilde{D}_n:=\frac{D-P_n}{\tilde{\varepsilon}_n}.
    \end{equation*}
    By a direct calculation we can derive that $\tilde{V}_n$ solves
    \begin{equation*}
        \begin{cases}
            -\Delta \tilde{V}_n(y) = \tau_n\left|\tilde{\varepsilon}_n y+P_n\right|^\mu \frac{f(\tilde{\varepsilon}_n^{-\frac{2}{q-2}} \tilde{V}_n(y))}{\tilde{\varepsilon}_n^{-\frac{2(q-1)}{q-2}} \tilde{V}_n^{q-1}(y)} \tilde{V}_n^{q-1}(y) -\tilde{\varepsilon}_n^2 \lambda_n\left|\tilde{\varepsilon}_n y+P_n\right|^\mu \tilde{V}_n(y) & \text { in } \tilde{D}_n, \\ 
            0< \tilde{V}_n \leq \tilde{V}_n(0)=1 & \text { in } \tilde{D}_n \cap B_{R_n}(0), \\ 
            \tilde{V}_n=0 & \text { on } \partial \tilde{D}_n.
        \end{cases}
    \end{equation*}
    Let $y = 0$ we have
    \begin{equation*}
        \tau_n \frac{f(\tilde{\varepsilon}_n^{-\frac{2}{q-2}})}{\tilde{\varepsilon}_n^{-\frac{2(q-1)}{q-2}}} |P_n|^\mu - \lambda_n \tilde{\varepsilon}_n^2 |P_n|^\mu = -\Delta \tilde{V}_n(0) \geq 0.
    \end{equation*}
    Using ($\mathbf{A_2}$) it follows that there exists $\tilde{\lambda} \in [0, a_0]$ such that up to subsequence, $\lambda_n \tilde{\varepsilon}_n^2 \rightarrow \tilde{\lambda} $ as $n \rightarrow+\infty$. In addition, we may suppose that $P_n \rightarrow P \in \bar{D}$. 

    Now if we denote $d_n = \operatorname{dist}\left(P_n, \partial D\right)$ and assume that 
    \begin{equation*}
        \frac{\tilde{\varepsilon}_n}{d_n} \rightarrow L \in[0,+\infty],
    \end{equation*}
    then using the elliptic regularity theory, we can deduce that, up to subsequence, $\tilde{V}_n \rightarrow \tilde{V}$ in $C_{\mathrm{loc}}^1(\bar{H})$, where $\tilde{V}$ solves the problem
    \begin{equation}\label{eq4.22}
        \begin{cases}
            -\Delta \tilde{V}+\tilde{\lambda}|P|^\mu \tilde{V}=a_0 |P|^\mu \tilde{V}^{q-1} & \text { in } H, \\ 
            0< \tilde{V} \leq \tilde{V}(0)=1 & \text { in } H, \\ 
            \tilde{V}=0 & \text { on } \partial H.
        \end{cases}
    \end{equation}
    and $H=\mathbb{R}^{M}$ if $L=0$, $H$ represents a half-space with $0 \in \bar{H}$ and $\operatorname{dist}(0, \partial H)=\frac{1}{L}$ if $L>0$. 

    To deduce the precise type of $H$, we need to use the Morse index information of $\tilde{V}$. Now we prove that the Morse index of $\tilde{V}$ is at most 2. By contradiction, suppose that $m(\tilde{V}) \geq 3$. Then there exist $\phi_1, \phi_2, \phi_3 \in C_0^{\infty}(H)$ orthogonal in $L^2(H)$ such that the quadratic form
    \begin{equation*}
        Q\left(\phi_i, \bar{V}\right):=\int_H \left(\left|\nabla \phi_i\right|^2+\tilde{\lambda}|P|^\mu \phi_i^2-a_0 |P|^\mu(q-1) \tilde{V}^{q-2} \phi_i^2 \right) \dif x<0.
    \end{equation*}
    Notice that if we introduce
    \begin{equation*}
        \phi_{i, n}(x):=\tilde{\varepsilon}_n^{-\frac{M-2}{2}} \phi_i\left(\frac{x-P_n}{\tilde{\varepsilon}_n}\right)
    \end{equation*}
    for $i= 1, 2, 3$, then $\phi_{1,n}, \phi_{2,n},\phi_{3,n}$ are orthogonal in $L^2(D)$. 
    Moreover, we define
    \begin{equation*}
        \tilde{D}_{n,R} :=\left\{y \in \tilde{D}_n : |\tilde{\varepsilon}_n^{-\frac{2}{q-2}} \tilde{V}_n(y)| < R \right\}, \quad \tilde{D}_{n,R}^c := \tilde{D}_n \backslash \tilde{D}_{n,R}.
    \end{equation*}
    where $R > 0$ is defined as in ($\mathbf{A_3}$). Then we have
    \begin{equation*}
        \begin{aligned}
            & \int_{\tilde{D}} |x|^\mu f^{\prime}(v_n(x)) \phi_{i,n}^2(x) \dif x \\
             = & \int_{\tilde{D}_n} |\tilde{\varepsilon}_n y + P_n|^\mu f^{\prime}(\tilde{\varepsilon}_n^{-\frac{2}{q-2}} \tilde{V}_n(y)) \tilde{\varepsilon}_n^2 \phi_i^2(y) \dif y \\
             = &  \int_{\tilde{D}_{n,R}} |\tilde{\varepsilon}_n y + P_n|^\mu f^{\prime}(\tilde{\varepsilon}_n^{-\frac{2}{q-2}} \tilde{V}_n(y)) \tilde{\varepsilon}_n^2 \phi_i^2(y) \dif y + \int_{\tilde{D}_{n,R}^c} |\tilde{\varepsilon}_n y + P_n|^\mu f^{\prime}(\tilde{\varepsilon}_n^{-\frac{2}{q-2}} \tilde{V}_n(y)) \tilde{\varepsilon}_n^2 \phi_i^2(y) \dif y \\
             := &  I_{1,n} + I_{2,n}.
        \end{aligned}
    \end{equation*}
    For $I_{1,n}$, by the definition of $\tilde{D}_{n,R}$ and the boundedness of $f^{\prime}$ on bounded sets, we have
    \begin{equation*}
        |I_{1,n}| \leq C \tilde{\varepsilon}_n^2 \int_{\tilde{D}_{n,R}} \phi_i^2(y) \dif y \leq C \tilde{\varepsilon}_n^2 \rightarrow 0 \text { as } n \rightarrow+\infty.
    \end{equation*}
    For $I_{2,n}$, by ($\mathbf{A_3}$) we have
    \begin{equation*}
        \begin{aligned}
            I_{2,n} & =\int_{\tilde{D}_{n,R}^c} |\tilde{\varepsilon}_n y + P_n|^\mu \frac{f^{\prime}(\tilde{\varepsilon}_n^{-\frac{2}{q-2}} \tilde{V}_n(y))}{\tilde{\varepsilon}_n^{-\frac{2(q-2)}{q-2}} \tilde{V}_n^{q-2}(y)} \tilde{V}_n^{q-2}(y) \phi_i^2(y) \dif y \\
            & \geq \int_{\tilde{D}_{n,R}^c} |\tilde{\varepsilon}_n y + P_n|^\mu a_1 \tilde{V}_n^{q-2}(y) \phi_i^2(y) \dif y \\
            & \rightarrow \int_H |P|^\mu a_1 \tilde{V}^{q-2}(y) \phi_i^2(y) \dif y \text { as } n \rightarrow+\infty.
        \end{aligned}
    \end{equation*}
    Thus we have
    \begin{equation*}
        \begin{aligned}
            Q\left(\phi_{i, n}, v_n\right) & :=\int_D\left(\left|\nabla \phi_{i, n}\right|^2+\lambda_n|x|^\mu \phi_{i, n}^2-\tau_n|x|^\mu f^{\prime}\left(v_n\right) \phi_{i, n}^2 \right)\dif x \\
            & =\int_D \left(\tilde{\varepsilon}_n^{-M}\left|\nabla \phi_i\left(\frac{x-P_n}{\tilde{\varepsilon}_n}\right)\right|^2+\lambda_n|x|^\mu \tilde{\varepsilon}_n^{-(M-2)} \phi_i^2\left(\frac{x-P_n}{\tilde{\varepsilon}_n}\right) \right)\dif x -\tau_n \int_D |x|^\mu f^{\prime}(v_n(x)) \phi_{i,n}^2(x) \dif x \\
            & = \int_{\tilde{D}_n}\left(\left|\nabla \phi_i(y)\right|^2+\lambda_n\left|\tilde{\varepsilon}_n y+P_n\right|^\mu \tilde{\varepsilon}_n^2 \phi_i^2(y)\right) \dif y- \tau_n\left|\tilde{\varepsilon}_n y+P_n\right|^\mu \frac{f^{\prime}\left(\tilde{\varepsilon}_n^{-\frac{2}{q-2}} \tilde{V}_n(y)\right)}{\tilde{\varepsilon}_n^{-\frac{2(q-2)}{q-2}} \tilde{V}^{q-2}_n(y)} \tilde{V}_n^{q-2}(y) \phi_i^2(y) \dif y \\
            & \rightarrow \int_H \left(\left|\nabla \phi_i\right|^2+\tilde{\lambda}|P|^\mu \phi_i^2-a_1 |P|^\mu \tilde{V}^{q-2} \phi_i^2\right) \dif y \\
            & \leq \int_H \left(\left|\nabla \phi_i\right|^2+\tilde{\lambda}|P|^\mu \phi_i^2 - a_0 |P|^\mu (q-1) \tilde{V}^{q-2} \phi_i^2\right) \dif y \\
            & <0.
        \end{aligned}
    \end{equation*}
    This contradicts the fact that $m(v_n) \leq 2$. 
    
    Since $N \geq 2$, the case $\tilde{\lambda} = 0$ is excluded in both cases of $H$ by the Liouville-type theorem stated in \cite{Bahri1992}, so we have $\tilde{\lambda} > 0$. Now we may use another Liouville-type theorem in \cite[Theorem 1.1]{Esposito2011} to conclude that $H = \mathbb{R}^{M}$. So we have
    \begin{equation*}
        \left(\frac{\tilde{\varepsilon}_n}{\varepsilon_n}\right)^2=\lambda_n \tilde{\varepsilon}_n^2 \rightarrow \tilde{\lambda} \in (0,a_0] \text { as } n \rightarrow+\infty.
    \end{equation*}
    Now we analyze the function $V_n$. By a similar calculation as $\tilde{V}_n$, we can deduce that $V_n$ satisfies
    \begin{equation*}
        \begin{cases}
            -\Delta V_n(y) = \tau_n\left|\varepsilon_n y+P_n\right|^\mu \frac{f(\varepsilon_n^{-\frac{2}{q-2}} V_n(y))}{\varepsilon_n^{-\frac{2(q-1)}{q-2}} V_n^{q-1}(y)} V_n^{q-1}(y) - \left|\varepsilon_n y+P_n\right|^\mu V_n(y) & \text { in } D_n, \\ 
            0< V_n \leq V_n(0)=\left(\frac{\tilde{\varepsilon}_n}{\varepsilon_n}\right)^{-\frac{2}{q-2}} & \text { in } D_n \cap B_{R_n \tilde{\varepsilon}_n / {\varepsilon_n}}(0), \\ 
            V_n=0 & \text { on } \partial D_n.
        \end{cases}
    \end{equation*}
    where $D_n=\frac{D-P_n}{\varepsilon_n}$ and $R_n \rightarrow+\infty$. By the same argument as $\tilde{V}_n$, we have $V_n \rightarrow V$ in $C_{\mathrm{loc}}^1(\bar{H})$ as $n \rightarrow+\infty$, where $V$ solves
    \begin{equation}\label{eq4.23}
        \begin{cases}
            -\Delta V+ |P|^\mu V=a_0|P|^\mu V^{q-1} & \text { in } H, \\ 
            0< V \leq V(0) & \text { in } H, \\ 
            V=0 & \text { on } \partial H.
        \end{cases}
    \end{equation}
    Moreover, using a similar argument as $\tilde{\varepsilon}_n, \tilde{V}_n, \tilde{V}$ before, we have $H = \mathbb{R}^{M}$ and $m(V) \leq \sup\limits_n m\left(v_n\right) \leq 2$, so 
    \begin{equation*}
        \frac{\varepsilon_n}{d_n} \rightarrow 0,
    \end{equation*}
    which proves (i) - (iii) of the lemma.  By \cite[Theorem 2.3]{Esposito2011} we have $V \rightarrow 0$ as $|x| \rightarrow +\infty$ and $V \in H^1(\mathbb{R}^{M})$. Therefore, by the well-known result of the uniqueness of the solution to \eqref{eq4.23} satisfying the above properties, we can deduce that $m(V)= 1$. As a consequence, there exists $\phi \in C_0^{\infty}\left(\mathbb{R}^{M}\right)$ with $\operatorname{supp} \phi \subset B_R(0)$ for some $R>0$, such that
    \begin{equation*}
        \int_{\mathbb{R}^{M}} \left|\nabla \phi\right|^2+ |P|^\mu \phi^2-a_0 |P|^\mu(q-1) V^{q-2} \phi^2 \dif x<0.
    \end{equation*}
    If we choose the function
    \begin{equation}
        \phi_n(x):=\varepsilon_n^{-\frac{M-2}{2}} \phi\left(\frac{x-P_n}{\varepsilon_n}\right),
    \end{equation}
    using similar argument as $\phi_{i, n}(x)$ we can deduce that \eqref{eq4.19} holds for $n$ large enough, which proves the assertion (iv). 
    
    Finally, it suffices to prove \eqref{eq4.21}. By a change of variable we have
    \begin{equation*}
        \begin{aligned}
        & \lim _{n \rightarrow \infty} \lambda_n^{M-\frac{q}{q-2}} \int _{B_{R \varepsilon_n\left(P_n\right)}} |x|^{\mu} v_n^q \dif x \\
        = & \lim _{n \rightarrow \infty} \lambda_n^{\frac{M}{2}-\frac{q}{q-2}} \int_{B_{R \varepsilon_n\left(P_n\right)}} \varepsilon_n^{-\frac{2 q}{q-2}} |x|^{\mu} V_n^q\left(\frac{x-P_n}{\varepsilon_n}\right) \dif x \\
        = & \lim _{n \rightarrow \infty} \lambda_n^{\frac{M}{2}-\frac{q}{q-2}} \int_{B_R(0)} \varepsilon_n^{M-\frac{2 q}{q-2}} |\varepsilon_n x +P_n|^{\mu} V_n^q(x) \dif x \\
        = & \lim _{n \rightarrow \infty} \int_{B_R(0)} |\varepsilon_n x +P_n|^{\mu} V_n^q(x) \dif x=\int_{B_R(0)} |P|^{\mu} V^q \dif x.
        \end{aligned}
    \end{equation*}
    The last step follows from the facts that $V_n \rightarrow V$ in $L^q_{\mathrm{loc}}(\mathbb{R}^{M})$ by the Sobolev compact embedding, and the dominated convergence theorem. The proof is complete.
\end{proof}

\textbf{Step 3:} We establish a global behavior of the sequence $\{v_n\}$.

\begin{lemma}\label{lem4.9}
    Assume that $\lambda_n \rightarrow+\infty$ as $n \rightarrow+\infty$. Then there exist at most two sequences of points $\left\{P_n^i\right\}, i=\{1, 2\}$, such that
    \begin{equation}\label{eq4.25}
        \lambda_n \operatorname{dist}\left(P_n^i, \partial D\right)^2 \rightarrow+\infty
    \end{equation}
    as $n \rightarrow+\infty$ and
    \begin{equation}\label{eq4.26}
        v_n\left(P_n^i\right)=\max _{D \cap B_{R_n \lambda_n^{-1/2}}\left(P_n^i\right)} v_n
    \end{equation}
    for some $R_n \rightarrow+\infty$ as $n \rightarrow+\infty$. If there exist two sequences of points $\left\{P_n^i\right\}, i=\{1, 2\}$ satisfying \eqref{eq4.25}-\eqref{eq4.27}, then we also have
    \begin{equation}\label{eq4.27}
        \lambda_n\left|P_n^1-P_n^2\right|^2 \rightarrow+\infty.
    \end{equation}
    Moreover, we have
    \begin{equation}\label{eq4.28}
        \lim _{R \rightarrow+\infty}\left(\limsup _{n \rightarrow+\infty}\left[\lambda_n^{-\frac{1}{q-2}} \max _{\left\{d_n(x) \geq R \varepsilon_n\right\}} v_n(x)\right]\right)=0,
    \end{equation}
    where $d_n(x)=\min \left\{\left|x-P_n^i\right|, i=1,2 \right\}$ is the distance function from the sequences $\left\{P_n^i\right\}$. 
\end{lemma}

\begin{proof} 
    First we choose $P_n^1$ such that $v_n\left(P_n^1\right)=\max \limits_D v_n$. If \eqref{eq4.28} is satisfied for $P_n^1$, using Lemma \ref{lem4.8} we can deduce that $P_n^1$ satisfies \eqref{eq4.25} and \eqref{eq4.26}, which means the claim holds for one sequence. Otherwise, if $P_n^1$ does not satisfy \eqref{eq4.28}, we suppose, for some $\delta>0$,  
    \begin{equation}\label{eq4.29}
        \limsup _{R \rightarrow+\infty}\left(\limsup _{n \rightarrow+\infty}\left[\varepsilon_n^{\frac{2}{q-2}} \max _{\left\{\left|x-P_n^1\right| \geq R \varepsilon_n \right\}} v_n\right]\right)=4 \delta>0.
    \end{equation}
    By Lemma \ref{lem4.8}, up to a subsequence we have
    \begin{equation}\label{eq4.30}
        \varepsilon_n^{\frac{2}{q-2}} v_n\left(\varepsilon_n y+P_n^1\right)=: V_n^1(y) \rightarrow V(y) \text { in } C_{\mathrm{loc}}^1\left(\mathbb{R}^{M}\right)
    \end{equation}
    as $n \rightarrow+\infty$. Since $V \rightarrow 0$ as $|y| \rightarrow+\infty$, we can find $R$ large so that
    \begin{equation}\label{eq4.31}
        V(y) \leq \delta, \quad \forall |y| \geq R.
    \end{equation}
    By \eqref{eq4.29}, up to take $R$ larger and up to a subsequence, we can also assume that
    \begin{equation}\label{eq4.32}
        \varepsilon_n^{\frac{2}{q-2}} \max _{\left\{|x-P_n^1| \geq R \varepsilon_n\right\}} v_n \geq 2 \delta.
    \end{equation}
    Since $v_n=0$ on $\partial D$, so we can find $P_n^2 \in D \backslash B_{R \varepsilon_n}\left(P_n^1\right)$ such that $v_n\left(P_n^2\right)= \max _{D \backslash B_{R \varepsilon_n}\left(P_n^1\right)} v_n$.

    Now we prove that \eqref{eq4.27} holds for $\left\{P_n^1, P_n^2\right\}$, i.e. $\frac{\left|P_n^2-P_n^1\right|}{\varepsilon_n} \rightarrow+\infty$. Suppose on the contrary that $\frac{\left|P_n^2-P_n^1\right|}{\varepsilon_n^1} \rightarrow R^{\prime} \geq R$, by \eqref{eq4.30} and \eqref{eq4.31} we would get
    \begin{equation*}
        \varepsilon_n^{\frac{2}{q-2}} v_n\left(P_n^2\right)=V_n^1\left(\frac{P_n^2-P_n^1}{\varepsilon_n}\right) \rightarrow V\left(R^{\prime}\right) \leq \delta,
    \end{equation*}
    in contradiction with \eqref{eq4.32}. Therefore, \eqref{eq4.27} does hold for $\left\{P_n^1, P_n^2\right\}$. 
    
    Now we set $\hat{\varepsilon}_n=v_n\left(P_n^2\right)^{-\frac{q-2}{2}}$ and $\hat{R}_n=\frac{1}{2} \frac{\left|P_n^2-P_n^1\right|}{\hat{\varepsilon}_n}$. By \eqref{eq4.32} we get $\hat{\varepsilon}_n \leq(2 \delta)^{-\frac{q-2}{2}} \varepsilon_n$, and then
    \begin{equation*}
        \hat{R}_n \geq \frac{(2 \delta)^{\frac{q-2}{2}}}{2} \frac{\left|P_n^2-P_n^1\right|}{\varepsilon_n} \rightarrow+\infty \quad \text { as } n \rightarrow+\infty.
    \end{equation*}
    We claim that
    \begin{equation*}
        v_n\left(P_n^2\right)=\max _{D \cap B_{\hat{R}_n\hat{\varepsilon}_n \left(P_n^2\right)}} v_n.
    \end{equation*}
    Indeed, since $\varepsilon_n<<\left|P_n^2-P_n^1\right|$ by \eqref{eq4.27}, for all $x \in B_{\hat{R}_n\hat{\varepsilon}_n}\left(P_n^2\right)$ we have
    \begin{equation*}
        \left|x-P_n^1\right| \geq\left|P_n^2-P_n^1\right|-\left|x-P_n^2\right| \geq \frac{1}{2}\left|P_n^2-P_n^1\right| \geq R \varepsilon_n
    \end{equation*}
    so $D \cap B_{\hat{R}_n\hat{\varepsilon}_n}\left(P_n^2\right) \subset D \backslash B_{R \varepsilon_n^1}\left(P_n^1\right)$ and the claim holds. Since $\hat{R}_n \rightarrow +\infty$ as $n \rightarrow+\infty$, by Lemma \ref{lem4.8} we also get that \eqref{eq4.25} and \eqref{eq4.26} hold true for $\left\{P_n^1, P_n^2\right\}$. 
    
    Now if \eqref{eq4.29} holds for $\left\{P_n^1, P_n^2\right\}$, we are done. Otherwise, we can iterate the above argument and attain three sequences $\{P_n^1\}$,$\{P_n^2\}$, $\{P_n^3\}$ so that \eqref{eq4.25}-\eqref{eq4.27} hold true, but \eqref{eq4.29} is not satisfied. We will show that this is impossible.

    For the sequence $P_n^i$, $i=1, 2, 3$, Using Lemma \ref{lem4.8}, there exists $\phi_n^i \in C_0^{\infty}(D)$ with $\operatorname{supp} \phi_n^i \subset B_{R \varepsilon_n}\left(P_n^i\right), R>0$, which satisfy 
    \begin{equation*}
        \int_D\left(\left|\nabla \phi_{n}^i \right|^2+\lambda_n|x|^\mu (\phi_{n}^i)^2-\tau_n|x|^\mu f^{\prime}\left(v_n\right) (\phi_{n}^i)^2 \right)\dif x<0 .
    \end{equation*}
    By \eqref{eq4.27}, $\phi_n^1, \ldots, \phi_n^3$ have disjoint compact supports for $n$ large, so we have $3 \leq \lim _{n \rightarrow+\infty} m\left(v_n\right)$, contradicting to $m(v_n) \leq 2$. Therefore, the argument must end after at most two iterations and the proof is complete.
\end{proof}

\textbf{Step 4:} We show the exponential decay of $\{v_n\}$ away from the local maximum points when $n$ is large.

\begin{lemma}\label{lem4.10}
    There exist $\gamma, C>0$ and $n_0 \in \mathbb{N}$ such that for all $n \geq n_0$ we have
    \begin{equation}\label{eq4.33}
        v_n(x) \leq C \lambda_n^{\frac{1}{q-2}} \sum_{i=1}^k e^{-\gamma \lambda_n^{1 / 2}\left|x-P_n^i\right|}, \quad \forall x \in D.
    \end{equation}
\end{lemma}

\begin{proof}
    By ($\mathbf{A_1}$) and ($\mathbf{A_2}$), there exists $C_1, C_2>0$ such that
    \begin{equation*}
        f(s) \leq C_1 s+ C_2 s^{q-1}, \quad \forall s \geq 0.
    \end{equation*}
    From \eqref{eq4.29}, there exists $R>0$ large enough such that for $n \geq n(R)$ large enough
    \begin{equation}\label{eq4.35}
        \lambda_n^{-\frac{1}{q-2}} \max _{\left\{d_n(x) \geq R \lambda_n^{-1 / 2}\right\}} v_n(x) < \left(\frac{1}{8 C_1 (q-1)}\right)^{\frac{1}{q-2}},
    \end{equation}
    where $d_n(x)=\min \left\{\left|x-P_n^i\right|: i=1,2\right\}$. Let
    \begin{equation*}
        A_n:=\left\{d_n(x) \geq R \lambda_n^{-1 / 2} \right\}
    \end{equation*}
    for $n \geq n(R)$. By \eqref{eq4.35}, for any $x \in A_n$ we have
    \begin{equation*}
            -\Delta v_n+\frac{\lambda_n}{2} |x|^\mu v_n = \tau_n |x|^\mu \left(\frac{f(v_n)}{v_n}-\frac{\lambda_n}{2 \tau_n}\right) v_n \leq \tau_n |x|^\mu \left(C_1 + C_2 v_n^{q-2} - \frac{\lambda_n}{2 \tau_n}\right) v_n \leq 0.
    \end{equation*}
    when $n$ is large enough. Now we introduce the function
    \begin{equation*}
        \phi_n^i(x):=e^{-\gamma \lambda_n^{1 / 2}\left|x-P_n^i\right|},
    \end{equation*}
    where $\gamma >0$ is a constant to be determined later. By direct calculation, for $\mu \geq 0$, we have
    \begin{equation*}
        \begin{aligned}
         \left(-\Delta+\frac{\lambda_n}{2}|x|^\mu \right) \phi_n^i(x) 
         = & \left(-\gamma^2 \lambda_n + \frac{\gamma (N-1) \lambda_n^{1 / 2}}{\left|x-P_n^i\right|} + \frac{\lambda_n}{2} |x|^\mu \right) e^{-\gamma \lambda_n^{1 / 2}\left|x-P_n^i\right|} \\
         = & \lambda_n \left(-\gamma^2 + \frac{\gamma (N-1)}{\lambda_n^{1 / 2} \left|x-P_n^i\right|} + \frac{1}{2} |x|^\mu \right) e^{-\gamma \lambda_n^{1 / 2}\left|x-P_n^i\right|} \\
         \geq & \lambda_n \left(-\gamma^2 + \frac{1}{2} a^\mu \right) e^{-\gamma \lambda_n^{1 / 2}\left|x-P_n^i\right|} \geq 0,
        \end{aligned}
    \end{equation*}
    in $A_n$ for $n$ large enough, provided we choose
    \begin{equation*}
        0<\gamma \leq \frac{1}{2} a^{\mu / 2}.
    \end{equation*}
    The case $\mu<0$ can be treated similarly if we choose $0<\gamma \leq \frac{1}{2} b^{\mu / 2}$. 
    Moreover, by $V(x) \rightarrow 0$ as $|x| \rightarrow+\infty$, it follows that for $R>0$ large enough
    \begin{equation*}
        \left. \left(e^{\gamma R} \phi_n^i(x)-\lambda_n^{-\frac{1}{q-2}} v_n(x)\right) \right\rvert_{\partial B_{R \lambda_n^{-1 / 2}}\left(P_n^i\right)} \rightarrow 1-V(R)>0
    \end{equation*}
    as $n \rightarrow+\infty$. Next, we define
    \begin{equation*}
        \phi_n:=e^{\gamma R} \lambda_n^{\frac{1}{q-2}} \sum_{i=1}^k \phi_n^i.
    \end{equation*}
    Considering the operator $L_n=-\Delta+\lambda_n|x|^\mu$, use the above results we can deduce that
    \begin{equation*}
        L_n\left(\phi_n-v_n\right) \geq 0 \quad \text { in } A_n
    \end{equation*}
    and $\phi_n-v_n \geq 0$ in $\partial A_n \cup \partial D$. Using \eqref{eq4.25}-\eqref{eq4.27}, we obtain
    \begin{equation*}
        \partial A_n=\bigcup_{i=1}^k \partial B_{R \lambda_n^{-1 / 2} }\left(P_n^i\right) \subset D.
    \end{equation*}
    So $\partial A_n \cup \partial D = \partial A_n$. Now we can use the comparison principle to deduce that for any $x \in A_n$,
    \begin{equation*}
        v_n(x) \leq \phi_n(x)=e^{\gamma R} \lambda_n^{\frac{1}{q-2}} \sum_{i=1}^k e^{-\gamma \lambda_n^{1 / 2}\left|x-P_n^i\right|}.
    \end{equation*}
    For $x \in D \backslash A_n$, by $\lambda_n \tilde{\varepsilon}_n^2 \rightarrow \tilde{\lambda}$, there exists $C>0$, such that
    \begin{equation*}
        v_n(x) \leq \max _D v_n= v_n\left(P_n\right)=\left(\tilde{\varepsilon}_n\right)^{-\frac{2}{q-2}} \leq C e^{\gamma R} \lambda_n^{\frac{1}{q-2}} \sum_{i=1}^k e^{-\gamma \lambda_n^{1 / 2} e_n^{-1 / 2}\left|x-P_n^i\right|}
    \end{equation*}
    when $n$ is large enough. Thus, taking a subsequence if necessary, for some $C>0$, \eqref{eq4.33} holds for any $n$ large enough and the proof is complete.
\end{proof}

\textbf{Step 5:} We will use the information obtained in the previous steps to prove the boundedness of the sequence $\{v_n\}$ when the energy level is bounded. 

First using the monotonicity of $\tilde{m}_{\tau_n}$, we get
\begin{equation*}
    E_1\left(w_1\right) \leq E_{\tau_n}\left(w_1\right) \leq \tilde{m}_{\tau_n} \leq \tilde{m}_{\frac{1}{2}},
\end{equation*}
which implies the boundedness of $\tilde{m}_{\tau_n}$. To prove Theorem \ref{thm1.2} (iii), it is crucial to prove the following lemma.

\begin{lemma}\label{lem4.11}
    Let $\left\{v_n\right\} \subset H_0^1(D)$ be a sequence of positive solutions to \eqref{eq4.17} for some $\left\{\lambda_n\right\} \subset \mathbb{R}$ and $\tau_n \rightarrow 1^{-}$. Assume that
    \begin{equation*}
        \int_{D} |x|^\mu v_n^2 \dif x= c^2 \quad \text { and } \quad m\left(v_n\right) \leq 2, \quad \forall n,
    \end{equation*}
    and assume the energy levels $\left\{\tilde{m}_n=E_{\tau_n}\left(v_n\right)\right\}$ is bounded. Then, the sequences $\left\{\lambda_n\right\} \subset$ $\mathbb{R}$ and $\left\{v_n\right\} \subset H_0^1(D)$ must be bounded. In addition, $\left\{v_n\right\} \subset {\S}_c$ is a bounded Palais-Smale sequence for $E$.
\end{lemma}

\begin{proof}
    Since $\left\{v_n\right\} \subset H_0^1(D)$ is a positive solution to \eqref{eq4.17} for each $n$, multiplying \eqref{eq4.17} by the first eigenfunction $\varphi_1$ of $-\Delta$ in $H_0^1(D)$ and integrating over $D$, we obtain
    \begin{equation*}
        \begin{aligned}
            \lambda_n \int_{D} |x|^\mu v_n \varphi_1 \dif x  = \int_{D} \nabla v_n \nabla \varphi_1 \dif x+\tau_n \int_{D} |x|^\mu f\left(v_n\right) \varphi_1 \dif x 
             \geq \int_{D} \nabla v_n \nabla \varphi_1 \dif x 
             = \lambda_1 \int_{D} v_n \varphi_1 \dif x,
        \end{aligned}
    \end{equation*}
    where $\lambda_1>0$ is the first eigenvalue of $-\Delta$ in $H_0^1(D)$. Using the fact that $|x|^{-\mu}$ is bounded from above in $D$, we can deduce that $\lambda_n$ is bounded from below.
    
    Now, suppose that on the contrary we have $\lambda_n \to +\infty$, up to a subsequence. Let $P_n^i$ denotes the local maximum point of the rescaled blow-up function $V_n^i$ as defined in Lemma \ref{lem4.8} , and $P^i$ denotes the corresponding limit point, where $i=1,2$.

    We first claim that, for any $R > 0$, we have
    \begin{equation}\label{eq5.1}
        \left|\lambda_n^{\frac{M}{2}-\frac{2}{q-2}} \int_{D} |x|^\mu v_n^2 \dif x-\sum_{i=1}^k \int_{B_R\left(0\right)}|\lambda_n^{-\frac{1}{2}}x + P_n^i|^\mu \left(V_n^i\right)^2 \dif x\right| \rightarrow+\infty .
    \end{equation}
    Indeed, we have     
    \begin{equation}\label{eq5.2}
        \lambda_n^{\frac{M}{2}-\frac{2}{q-2}} \int_D |x|^\mu u_n^2 \dif x=\lambda_n^{\frac{M}{2}-\frac{2}{q-2}} c^2 \rightarrow+\infty
    \end{equation}
    while using Lemma \ref{lem4.8} (v), we deduce 
    \begin{equation*}
        \sum_{i=1}^k \int_{B_R\left(0\right)}|\lambda_n^{-\frac{1}{2}}x + P_n^i|^\mu \left(V_n^i\right)^2 \dif x \rightarrow \sum_{i=1}^k \int_{B_R\left(0\right)} |P^i|^\mu \left(V^i\right)^2 \dif x
    \end{equation*}
    which together with \eqref{eq5.2} shows that \eqref{eq5.1} holds.

    On the other hand, from Lemma \ref{lem4.10} we may deduce that
    \begin{equation*}
        \begin{aligned}
            & \left|\lambda_n^{\frac{M}{2}-\frac{2}{q-2}} \int_{D} |x|^\mu v_n^2 \dif x-\sum_{i=1}^k \int_{B_R\left(0\right)}|\lambda_n^{-\frac{1}{2}}x + P_n^i|^\mu \left(V_n^i\right)^2 \dif x\right| \\
            = &\lambda_n^{\frac{M}{2}-\frac{2}{q-2}} \left| \int_{D} |x|^\mu v_n^2 \dif x - \sum_{i=1}^k \int_{B_{R \lambda_n^{-1 / 2}}\left(P_n^i\right)}|x|^\mu v_n^2 \dif x \right| \\
            = & \lambda_n^{\frac{M}{2}-\frac{2}{q-2}}  \int_{D \backslash \bigcup_j B_{R \lambda_n^{-1 / 2} }\left(P_n^j\right)} |x|^\mu v_n^2 \dif x \\
            \leq & C \lambda_n^{\frac{M}{2}}  \sum_{i=1}^k \int_{D \backslash \bigcup_j B_{R \lambda_n^{-1 / 2}} \left(P_n^j\right)} e^{-\gamma \lambda_n^{1 / 2} \left|x-P_n^i\right|} \dif x \\
            \leq & C \lambda_n^{\frac{M}{2}} \int_{R \lambda_n^{-1 / 2} }^{+\infty} e^{-2 \gamma \lambda_n^{1 / 2}  y} d y \\
            \leq & C \int_R^{+\infty} e^{-2 \gamma z} d z \\
            \leq & C e^{-2 \gamma R}.
        \end{aligned}
    \end{equation*}
    This implies that
    \begin{equation*}
        \limsup _{n \rightarrow+\infty}\left|\lambda_n^{\frac{M}{2}-\frac{2}{q-2}} \int_{D} |x|^\mu v_n^2 \dif x-\sum_{i=1}^k \int_{B_R\left(0\right)}|\lambda_n^{-\frac{1}{2}}x + P_n^i|^\mu \left(V_n^i\right)^2 \dif x\right| \leq C e^{-2 \gamma R},
    \end{equation*}
    which provides a contradiction with \eqref{eq5.1}. Therefore, the sequence $\left\{\lambda_n\right\}$ is bounded.

    Now we prove the boundedness of $\left\{v_n\right\}$ in $H_0^1(D)$. Suppose on the contrary that $\left\|v_n\right\| \rightarrow+\infty$ as $n \rightarrow+\infty$. Multiplying \eqref{eq4.17} by $v_n$ and integrating over $D$, we obtain
    \begin{equation*}
        \int_{D} \left|\nabla v_n\right|^2 \dif x + \lambda_n \int_{D} |x|^\mu v_n^2 \dif x = \tau_n \int_{D} |x|^\mu f\left(v_n\right) v_n \dif x.
    \end{equation*}
    Thus by ($\mathbf{A_1}$) and ($\mathbf{A_2}$), for $\mu \geq 0$ we have
    \begin{equation*}
        \begin{aligned}
            \int_{D} \left|\nabla v_n\right|^2 \dif x & = \tau_n \int_{D} |x|^\mu f\left(v_n\right) v_n \dif x - \lambda_n \int_{D} |x|^\mu v_n^2 \dif x \\
            & = \tau_n \int_{D} |x|^\mu f\left(v_n\right) v_n \dif x - \lambda_n c^2 \\
            & \leq \tau_n b^\mu (C_1 ||v_n||_{L^2(D)}^2 + C_2 ||v_n||_{L^q(D)}^q) - \lambda_n c^2 \\
            & \leq \tau_n b^\mu |D| (C_1 ||v_n||_{L^{\infty}(D)}^2 + C_2 ||v_n||_{L^{\infty}(D)}^{q}) - \lambda_n c^2. 
        \end{aligned}
    \end{equation*}
    Thus by the boundedness of $\left\{\lambda_n\right\}$, we get $\left\|v_n\right\|_{L^{\infty}(D)} \rightarrow+\infty$ as $n \rightarrow+\infty$. The case $\mu<0$ can be treated similarly.
    Let $P_n \in D$ be such that $v_n\left(P_n\right)=\left\|v_n\right\|_{L^{\infty}(D)}$. Define $\bar{\varepsilon}_n=v_n\left(P_n\right)^{-\frac{q-2}{2}}$ and the blow-up sequence
    \begin{equation*}
        \bar{v}_n(y):=\bar{\varepsilon}_n^{\frac{2}{q-2}} v_n\left(\bar{\varepsilon}_n y+P_n\right).
    \end{equation*}
    Then $\bar{v}_n$ satisfies
    \begin{equation*}
        \begin{cases}
            -\Delta \bar{v}_n  +\lambda_n \bar{\varepsilon}_n^2 |\bar{\varepsilon}_n y+P_n|^\mu \bar{v}_n = \tau_n |\bar{\varepsilon}_n y+P_n|^\mu \frac{f\left(\bar{\varepsilon}_n^{-\frac{2}{q-2}} \bar{v}_n\right)}{\left(\bar{\varepsilon}_n^{-\frac{2}{q-2}} \bar{v}_n\right)^{q-1}} \bar{v}_n^{q-1} \quad \text { in } D_n, \\
            \bar{v}_n  >0 \quad \text { in } D_n, \\
            \bar{v}_n  =0 \quad \text { on } \partial D_n,
        \end{cases}
    \end{equation*}
    where $D_n=\left\{y \in \mathbb{R}^{M}: \bar{\varepsilon}_n y+P_n \in D\right\}$. Similar to the proof of Lemma \ref{lem4.8}, up to a subsequence we may assume that $\bar{v}_n \rightarrow \bar{v}$ in $C_{\mathrm{loc}}^1\left(H\right)$ as $n \rightarrow+\infty$, where $H$ is either $\mathbb{R}^{M}$ or a half space of $\mathbb{R}^{M}$. Moreover, $\bar{v}$ satisfies
    \begin{equation*}
        \begin{cases} 
            -\Delta \bar{v} = a_0 |P|^\mu \bar{v}^{q-1} \quad \text { in } H, \\
            \bar{v}  >0 \quad \text { in } H, \\
            \bar{v}  =0 \quad \text { on } \partial H,
        \end{cases}
    \end{equation*}
    where $P=\lim \limits_{n \rightarrow+\infty} P_n \in \bar{D}$, and $m(\bar{v}) \leq 2$. By the Liouville-type results in \cite{Bahri1992, Farina2007} we know that such problem does not admit any positive solution with finite Morse index, which is a contradiction. Therefore, the sequence $\left\{v_n\right\}$ is bounded in $H_0^1(D)$.  
\end{proof}

\begin{proof}[Proof of Theorem \ref{thm1.2} (iii): The mountain pass type solution]
    Using Lemma \ref{lem4.6},  we can deduce that for any $c \in \left(0, c_2\right)$, there exists a sequence $\tau_n \rightarrow 1^{-}$ and a corresponding sequence $\left\{v_n\right\} \subset \mathcal{S}_c$ which are mountain pass type critical points of $E_{\tau_n}$, such that $E_{\tau_n}\left(v_n\right)=\tilde{m}_{\tau_n}$ and $m\left(v_n\right) \leq 2$. Now we may apply Lemma \ref{lem4.11}, it follows that $\left\{v_n\right\}$ is a bounded Palais-Smale sequence for $E$. Then by Lemma \ref{lem4.2} for the functional $E_{\tau_n}$, we obtain the conclusion.
\end{proof}

\section{Comments and open problems}

In this paper, we have studied the existence of normalized solutions for the nonlinear Schrödinger equation with a large class of nonlinearities involving the Sobolev supercritical growth on an annulus in some particular dimensions. A natural and interesting question is whether these results can be extended to other cases. We will give some comments on these problems below, which will be discussed in our following works.

(i) In \cite{Ruf2014} the authors have given a reduction method for any even dimension $N \geq 4$. Indeed, there exists the following higher dimensional Hopf fibrations:
\begin{equation*}
   \begin{tikzcd}[column sep=large, row sep=large]
    S^{1} \arrow[r] & S^{2n+1} \arrow[d, "\pi"] \\
    & \mathbb{CP}^{\,n}
    \end{tikzcd}
    \qquad
    \begin{tikzcd}[column sep=large, row sep=large]
    S^{3} \arrow[r] & S^{4n+3} \arrow[d, "\pi"] \\
    & \mathbb{HP}^{\,n}
    \end{tikzcd}
    \qquad
    \begin{tikzcd}[column sep=large, row sep=large]
    S^{7} \arrow[r] & S^{15} \arrow[d, "\pi"] \\
    & \mathbb{OP}^{\,1}
    \end{tikzcd} 
\end{equation*}
Using these fibrations, the problem \eqref{eq:RPP} can be reduced to the following problem:
\begin{equation}\label{eq:5.1}
    \begin{cases}
        -\Delta v+ \lambda \frac{1}{(q s)^{2-\frac{2+\alpha}{q}}} v = \frac{1}{(q s)^{2-\frac{2+\alpha}{q}}} f(v) & \text { in } \mathcal{M}, \\
        v=0 & \text { on } \partial \mathcal{M}, \\
        \int_{\mathcal{M}} \frac{1}{(q s)^{2-\frac{2+\alpha}{q}}} v^2 \dif x = c^2,
    \end{cases}
\end{equation}
where $\mathcal{M}= (c,d) \times_g \mathbb{B}$ is a warped product manifold of dimension $m$ with $\mathbb{B}=\mathbb{CP}^{\,n}, \mathbb{HP}^{\,n}, \mathbb{OP}^{\,1}$, the warped function $g(s)= \beta s$ with $s = \pi_1(x) \in (c,d)$ denotes the projection onto the first coordinate of $\mathcal{M}$, and $\beta = \frac{n-2}{m-2}$ (see \cite{Ruf2014, Petersen2016} for more details). Notice that when $N=4,8,16$ the manifold $\mathcal{M}$ is flat and these cases have been studied in this paper. However, in other cases the manifold $\mathcal{M}$ is not flat, which brings some new difficulties:
\begin{itemize}
    \item Due to the non-homogeneous properties of the metric on $\mathcal{M}$, the prove of the local minimizer and the mountain pass geometry of the energy functional corresponding to \eqref{eq:5.1} will be more complicated.
    \item The blow-up analysis is more complicated due to the curvature of the manifold $\mathcal{M}$. 
\end{itemize}

(ii) Another reduction method for any dimension $N = 2m$ with $m \geq 2$ has been given in \cite{Pacella2014}. If one consider $\mathbb{R}^{2m}$ as the product of two copies $\mathbb{R}^m \times \mathbb{R}^m$ and consider the doubly symmetric function in $\mathbb{R}^{2m}$:
\begin{equation*}
    u\left(y_1, y_2\right)=u\left(\left|y_1\right|,\left|y_2\right|\right), \quad y_i \in \mathbb{R}^m, i=1,2,
\end{equation*}
then the problem \eqref{eq:RPP} can be reduced to the following problem in $\mathbb{R}^{m+1}$:
\begin{equation}\label{eq:5.2}
    \begin{cases}
        -\Delta v+ \lambda \frac{1}{2|z|} v = \frac{1}{2|z|} f(v) & \text { in } D, \\
        v=0 & \text { on } \partial D, \\
        \int_{D} \frac{1}{2|z|} v^2 \dif z = c^2,
    \end{cases}
\end{equation}
where $D = \left\{x \in \mathbb{R}^{m+1}: a^2/2 < |z| < b^2/2\right\}$.
Notice that in this case the function space of $v$ is \begin{equation*}
    H_{\#}(D)=\left\{v \in H_0^1(D): v\left(z_1, \ldots, z_m, z_{m+1}\right)=v\left(\sqrt{z_1^2+\cdots+z_m^2},\left|z_{m+1}\right|\right)\right\},
\end{equation*}
which makes the blow-up analysis more complicated. Indeed, if one wants to use the blow-up analysis in Section 4, due to the axial symmetry of the function, if one local maximum point $P_n$ is not at the $z_{m+1}$-axis, there must be infinite many local maximum points, which makes the energy estimate impossible. A method to overcome this difficulty is to consider the function $v$ as a function in the two dimensional space $\mathbb{R}^2$, i.e., $v(y)=v\left(r,s\right)$ with $r = \sqrt{z_1^2+\cdots+z_m^2}$ and $s = z_{m+1}$. Then the problem will be transformed to the following problem in two dimension
\begin{equation}\label{eq:5.3}
    \begin{cases}
        -v_{rr}-v_{ss}-\frac{1}{r} v_r -\frac{1}{s} v_s + \lambda \frac{1}{2\sqrt{r^2+s^2}} v = \frac{1}{2\sqrt{r^2+s^2}} f(v) & \text { in } \Omega, \\
        v > 0 & \text { in } \Omega, \\
        v = 0 & \text { on } D_1 \cup D_2, \\
        \partial v / \partial \nu = 0 & \text { on } D_3 \cup D_4, \\
        \int_{\Omega} \frac{1}{2\sqrt{r^2+s^2}} v^2 r \dif r \dif s = c^2,
    \end{cases}
\end{equation}
where $\Omega = \left\{(r,s) \in \mathbb{R}^2: a^2/2 < \sqrt{r^2+s^2} < b^2/2, r>0, s>0 \right\}$with boundary $\partial \Omega = D_1 \cup D_2 \cup D_3 \cup D_4$ as shown in the following figure.
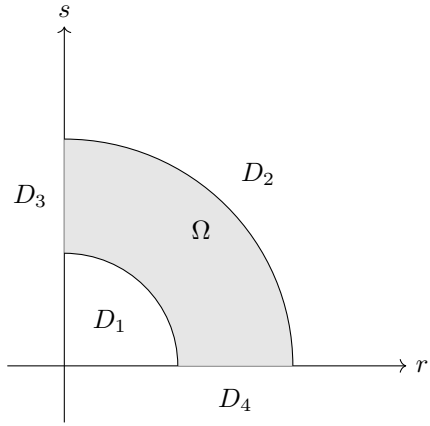
\begin{figure}[ht]
    \centering
    \begin{tikzpicture}[scale=1.5]
        \draw[->] (-0.5, 0) -- (3, 0) node[right] {$r$};
        \draw[->] (0, -0.5) -- (0, 3) node[above] {$s$};
        
        \draw[thick] (0, 2) arc[start angle=90, end angle=0, radius=2];
        \draw[thick] (0, 1) arc[start angle=90, end angle=0, radius=1];
        
        \fill[gray!20] (0, 1) arc[start angle=90, end angle=0, radius=1] -- (2, 0) arc[start angle=0, end angle=90, radius=2] -- cycle;
        
        \node at (1.2, 1.2) {$\Omega$};
        \node at (0.4, 0.4) {$D_1$};
        \node at (1.7, 1.7) {$D_2$};
        \node at (-0.3, 1.5) {$D_3$};
        \node at (1.5, -0.3) {$D_4$};
    \end{tikzpicture}
    \caption{The domain $\Omega$ in the $(r,s)$-plane.}
    \label{fig:domain}
\end{figure}
Problem \eqref{eq:5.3} is a two dimensional problem with mixed boundary conditions, so a novel blow-up analysis with mixed boundary conditions need to be developed to deal with this problem. This will involve careful consideration of the behavior of solutions near the boundary, the interaction between the Dirichlet and Neumann parts of the boundary, and the new exponential decay estimates away from the blow-up points, which are all challenging problems. We will discuss these problems in the following works.

(iii) Finally, it is also interesting to study the problem \eqref{eq:DP} in the exterior domain, i.e., $\Omega = \mathbb{R}^N \backslash \overline{B_a(0)}$ with $a>0$, $N=4,8,16$. The reduction method in Section 2.1 is still valid in this case, so the problem may be reduced to the following problem:
\begin{equation*}
    \begin{cases}
        -\Delta v + \frac{\lambda}{2 |x|} v = \frac{1}{2 |x|} v^{p-1} & \text { in } \mathbb{R}^M \backslash \overline{B_{a^2/2}(0)}, \\
        v=0 & \text { on } \partial B_{a^2/2}(0), \\
        \int_{\mathbb{R}^M \backslash \overline{B_{a^2/2}(0)}} \frac{1}{2 |x|} v^2 \dif x = c^2.
    \end{cases}
\end{equation*}
where $M=3,5,9$ respectively for $N=4,8,16$. This problem corresponds to the problem \eqref{eq:RGP} with $\mu = -1$, $f(v) = v^{p-1}$, and $b = +\infty$. However, since the domain is unbounded, it brings some new difficulties:
\begin{itemize}
    \item The compactness of the Sobolev embedding $H_0^1\left(\mathbb{R}^M \backslash \overline{B_{a^2/2}(0)}\right) \hookrightarrow L^p\left(\mathbb{R}^M \backslash \overline{B_{a^2/2}(0)}\right)$ for $2 \leq p < 2^*$ is lost, so some new techniques need to be developed to overcome this difficulty, even in the case $2 < p < 2^{\#}_M$.
    \item The eigenfunction $\varphi_1$ of $-\Delta$ in $H_0^1\left(\mathbb{R}^M \backslash \overline{B_{a^2/2}(0)}\right)$ does not exist, so the local minimizer structure and the mountain pass geometry are no longer clear. To overcome this difficulty, one may consider the problem in a large annulus $A_R$ with radius $R> a^2/2$ and then let $R \rightarrow +\infty$, like the method in \cite{Bartsch2024}. However, to pass to the limit $R \rightarrow +\infty$, firstly the uniform boundedness of the $H_0^1\left(A_R\right)$-norm of the solutions is crucial, which is very challenging to obtain. The blow-up analysis in Section 4 may be helpful to deal with this problem, but some new difficulties may arise due to the varying and unbounded domain.
\end{itemize}

\section*{Acknowledgments}
This work was partially developed while the authors were pursuing their doctoral degrees at the Academy of Mathematics and Systems Science, Chinese Academy of Sciences.
All authors would like to express their gratitude to Professor Shujie Li, Professor Yanheng Ding, and Professor Chong Li for their valuable suggestions and insightful discussions.

\begin{appendix}
    \section{The reduction process}

    In this appendix, we give the detailed reduction process from \eqref{eq2.1} to \eqref{eq:2.2} when $N=4,8,16$. First we can write the Laplace operator in $\mathbb{R}^N$ in spherical coordinates
    \begin{equation*}
        \Delta u = u_{r r} + \frac{N-1}{r} u_r + \frac{1}{r^2} \Delta_{S^{N-1}} u.
    \end{equation*}
    Now we consider the function space $H_{0,G}^1(A)$ defined in Section 2.1. By \cite{Ruf2014}, the Hopf quotient maps are harmonic morphisms, which implies that the Laplace-Beltrami operator is preserved under the group action, i.e., respectively for $N = 4,8,16$,
    \begin{equation*}
        \Delta_{S^3} \rightarrow \Delta_{\mathbb{CP}^{\,1}}, \quad \Delta_{S^7} \rightarrow \Delta_{\mathbb{HP}^{\,1}}, \quad \Delta_{S^{15}} \rightarrow \Delta_{\mathbb{OP}^{\,1}}. 
    \end{equation*}
    We want to show that the Laplace operator of a $G$-invariant function $u \in H_{0,G}^1(A)$ can be reduced to the Laplace operator of a function $v$ defined on the corresponding projective line $\mathbb{FP}^{\,1}$, i.e. we want to show 
    \begin{equation*}
        \Delta v = v_{s s} + \frac{M-1}{s} v_s + \zeta \frac{1}{s^2} \Delta_{\mathbb{FP}^{\,1}} v
    \end{equation*}
    for some constant $\zeta >0$, $s \in (c,d)$ corresponds to the radial coordinate $r \in (a,b)$ and $M = 3,5,9$ respectively for $N = 4,8,16$.
    
    Notice that if we set $s = \frac{1}{2} r^2$ and $v(s, \cdot) = u(r, \cdot)$, then the radial part of the Laplace operator can be transformed as
    \begin{equation*}
        u_{r r} + \frac{N-1}{r} u_r = r^2 [v_{s s} + \frac{M-1}{s} v_s].
    \end{equation*}
    Thus, 
    \begin{equation*}
        r^2 [v_{s s} + \frac{M-1}{s} v_s + \frac{1}{r^4} \Delta_{S^{N-1}} v] = 2s [v_{s s} + \frac{M-1}{s} v_s + \frac{1}{4 s^2} \Delta_{S^{N-1}} v].
    \end{equation*}
    So the Hopf reduction implies that the reduced Laplace operator in $(s, \mathbb{FP}^{\,1})$ coordinates is
    \begin{equation*}
        \Delta v = v_{s s} + \frac{M-1}{s} v_s + \frac{1}{4 s^2} \Delta_{\mathbb{FP}^{\,1}} v.
    \end{equation*}
    Using the stereographic projection, we can identify $\mathbb{FP}^{\,1}$ with $S^{M-1}$, so the Laplace-Beltrami operator $\Delta_{\mathbb{FP}^{\,1}}$ can be identified with $\Delta_{S^{M-1}}$. Thus, the reduced problem is
    \begin{equation*}
        \begin{cases}
            -\Delta v + \lambda \frac{1}{(2 |x|)^{2-\frac{2+\alpha}{2}}} v = \frac{1}{(2 |x|)^{2-\frac{2+\alpha}{2}}} f(v) & \text { in } D, \\
            v=0 & \text { on } \partial D, \\
            \int_{D} \frac{1}{(2 |x|)^{2-\frac{2+\alpha}{2}}} v^2 \dif x = \hat{c}^2,
        \end{cases}
    \end{equation*}
    where $D = \left\{x \in \mathbb{R}^M: a^2/2 < |x| < b^2/2\right\}$ and $\hat{c} = c\omega_{N-M}^{-\frac{1}{2}}$ with $\omega_{N-M}$ being the volume of the unit sphere in $\mathbb{R}^{N-M}$. The reduction process is complete.
\end{appendix}

\noindent
\textbf{Jian Liang}
\\[6pt]
School of Mathematical Sciences,
Laboratory of Mathematics and
Complex Systems, MOE,
Beijing Normal University, 
100875 Beijing, China
\\[6pt]
E-mail: liangjian2020@amss.ac.cn

\vspace{0.5cm}

\noindent
\textbf{Hua-Yang Wang} 
\\[6pt]
School of Mathematical Sciences,
Laboratory of Mathematics and
Complex Systems, MOE,
Beijing Normal University, 
100875 Beijing, China
\\[6pt]
E-mail:wanghuayang@amss.ac.cn

\end{document}